\documentclass[11pt,a4paper]{amsart}
\usepackage[utf8]{inputenc}
\usepackage{faktor}
\usepackage[T1]{fontenc}
\usepackage{lmodern}
\usepackage[english]{babel}
\usepackage{amsmath}
\usepackage{xfrac}
\usepackage{esint}
\usepackage{stmaryrd,mathrsfs,bm,amsthm,yfonts,amssymb,color,braket,booktabs,graphicx,graphics,amsfonts}
\usepackage{latexsym,microtype,indentfirst,hyperref}
\usepackage{xcolor}
\usepackage{courier}
\usepackage[colorinlistoftodos]{todonotes}

\usepackage{graphicx,color,longtable}

\usepackage{amsmath,amsfonts,amssymb,amsthm,mathrsfs,amsopn,mathtools,faktor,braket,stmaryrd}
\usepackage{tikz-cd}
\usepackage{esint}
\usepackage{latexsym}
\usepackage[numbers]{natbib}

\usepackage{hyperref}

\usepackage{bbm}

\usepackage{latexsym}

\RequirePackage[left=2cm,top=1.1cm,bottom=1.3cm,right=2cm,includefoot,includehead]{geometry}

\DeclarePairedDelimiter{\abs}{|}{|}
\DeclarePairedDelimiter{\coup}{(}{)}				
\DeclarePairedDelimiter{\cquad}{[ }{]}
\DeclarePairedDelimiter{\norm}{\lVert}{\rVert}
\DeclarePairedDelimiter{\scal}{\langle}{\rangle}

\title[Stability for anisotropic nearly umbilical and quasi Einstein hypersurfaces]{Absence of bubbling phenomena for non convex anisotropic nearly umbilical and quasi Einstein hypersurfaces}
\author[A. De Rosa]{Antonio De Rosa}
\address{A. De Rosa: Department of Mathematics, University of Maryland, 4176 Campus Dr, College Park, MD 20742, USA}
\email{anderosa@umd.edu}

\author[S. Gioffrè]{Stefano Gioffrè}
\address{S. Gioffrè: Institut f\"ur Mathematik, Universit\"at Z\"urich, Winterthurerstrasse 190, CH-8057 Z\"urich, Switzerland}
\email{stefano.gioffre@math.uzh.ch}

\newtheorem{teo}{Theorem}[section]
\newtheorem{prop}[teo]{Proposition}
\newtheorem{cor}[teo]{Corollary}
\newtheorem{lemma}[teo]{Lemma}
\newtheorem{remark}[teo]{Remark}
\newtheorem{defi}[teo]{Definition}
\numberwithin{equation}{section}

\newcommand{\N}{\mathbb{N}}
\newcommand{\R}{\mathbb{R}}
\newcommand{\B}{\mathbb{B}}

\newcommand{\erre}{\mathbb{R}}
\newcommand{\esse}{\mathbb{S}}
\newcommand{\debto}{\rightharpoonup}
\newcommand{\fromto}[2]{\colon #1 \longrightarrow #2}
\newcommand{\daA}[2]{\colon #1 \longrightarrow #2}

\newcommand{\restr}[2]{\left. #1 \right|_{#2}}

\newcommand{\W}{\mathcal{W}}
\newcommand{\PF}{\mathcal{F}}

\newcommand{\Sdot}{\mathring{S}}
\newcommand{\hdot}{\mathring{h}}
\newcommand{\mR}{\mathcal{R}}

\newcommand{\tRic}{\mathring{\Ric}}

\newcommand{\Scal}{R}
\newcommand{\nomizu}{\owedge}
\DeclareMathOperator{\grad}{grad}

\renewcommand{\epsilon}{\varepsilon}
\renewcommand{\phi}{\varphi}
\renewcommand{\theta}{\vartheta}

\renewcommand{\S}{\mathbb{S}}
\newcommand{\gf}{\mathfrak{f}}

\newcommand{\gu}{\mathfrak{u}}

\newcommand{\Gammag}{_{g}\Gamma}

\DeclareMathOperator{\Ric}{Ric}

\DeclareMathOperator{\divv}{div}
\DeclareMathOperator{\Vol}{Vol}
\DeclareMathOperator{\hd}{HD}

\DeclareMathOperator{\Id}{Id}
\DeclareMathOperator{\diam}{diam}
\DeclareMathOperator{\Lip}{Lip}
\DeclareMathOperator{\Riem}{Riem}
\DeclareMathOperator{\tr}{tr}

\DeclareMathOperator{\Graph}{Graph}

\date{}
\makeindex
\allowdisplaybreaks

\begin{document}

\begin{abstract}
We prove that, for every closed (not necessarily convex) hypersurface $\Sigma$ in $\mathbb{R}^{n+1}$ and every $p>n$, the $L^p$-norm of the trace-free part of the anisotropic second fundamental form controls from above the $W^{2, \, p}$-closeness  of $\Sigma$ to the Wulff shape. In the isotropic setting, we provide a simpler proof. This result is sharp since in the subcritical regime $p\le n$, the lack of convexity assumptions may lead in general to bubbling phenomena. 
Moreover, we obtain a stability theorem for quasi Einstein (not necessarily convex) hypersurfaces and we improve the quantitative estimates in the convex setting.
\end{abstract}
\maketitle

\section{Introduction}\label{Intr}
The umbilical theorem, \cite[Lemma 1, p. 8]{Spivak99}, is a rigidity result which states that: given a closed, connected and smooth hypersurface $\Sigma$ of $\R^{n+1}$, if $\Sigma$ is umbilical, i.e. the trace free part of the second fundamental form is constantly equal to $0$, then $\Sigma$ is homothetic to a sphere.  The stability of this result has been addressed in \cite{DLM,DLC0,GiofSec,Daniel,JR1,JR2,JR3}, and produced important applications in the foliation of asymptotically flat three–manifolds by surfaces of prescribed mean curvature (see \cite{LM,LMS,Met}).

The umbilical theorem holds also in the anisotropic setting: in \cite{HeLi} it is shown that the only smooth closed hypersurface with anisotropic second fundamental form which is a constant multiple of the identity is the Wulff shape, see also \cite{DRKS} for the low-regularity case of finite perimeter sets. In \cite[Theorem 1.2]{GiofDeR}, the authors have recently proved qualitative and quantitative stability for the anisotropic rigidity result. Namely, given $p\in (1, \, +\infty)$ and $\Sigma$ a closed hypersurface in $\R^{n+1}$, which is the boundary of a convex, open set, then the $W^{2, \, p}$-closeness of $\Sigma$ to the Wulff shape is controlled by the $L^p$-norm of the trace-free part of the anisotropic second fundamental form.
For $n\geq 3$, the convexity assumption on the hypersurface $\Sigma$ is a necessary condition in order to avoid bubbling phenomena, as observed with a counterexample in \cite[Appendix A]{GiofDeR}. In this respect, in the recent paper \cite{DMMN}, it is proven that if $\Sigma$ is a closed hypersurface (not necessarily convex) with anisotropic mean curvature $L^2$-close to a constant, then $\Sigma$ is $L^1$-close to a finite union of Wulff shapes, extending the seminal work \cite{cir}.

The aim of this paper is to show that in the supercritical regime $p>n$, the convexity assumption on $\Sigma$ can be dropped. This problem was open also for the area funtional. In Section \ref{Isot} we provide a simpler proof in the isotropic case, while in Section \ref{aniso} we give the general proof for the anisotropic setting. 
Moreover, in Section \ref{Ricci}, we prove a similar theorem for non convex, quasi Einstein hypersurfaces. If $n \ge 3$ and $\Sigma$ is an Einstein closed hypersurface in $\erre^{n+1}$, it is well-known that it must be a round sphere. We prove that if an hypersurface is quasi Einstein in an $L^p$-sense, then it is $W^{2, \, p}$-close to a sphere with a quantitative estimate.

\subsection*{Nearly umbilical hypersurfaces}
In order to state our main results, we introduce some notation. We consider a smooth anisotropic function defined on the $n$-sphere:
$$F \daA{\esse^n}{(0, \, \infty)}.$$
For every  closed smooth hypersurface $\Sigma$ in $\erre^{n+1}$, we define its \textit{anisotropic surface energy} as
\begin{equation*}
\PF (\Sigma):= \int_\Sigma F(\nu_\Sigma) \, dV,
\end{equation*}
where $\nu_\Sigma$ will denote throughout the paper the outer normal vector field associated to $\Sigma$.
 In particular, the isotropic surface energy $\Vol_n (\Sigma)$ corresponds to the energy $\PF (\Sigma)$  associated to the function $F\equiv 1$. An increasing interest has been recently devoted to anisotropic energy functionals \cite{DDG2017,DPDRG2018,DDG,DDH,DeRosa,DRK}.

In \cite{Taylor}, J.\ Taylor has proved that the isoperimetrical shape, i.e. the solution of the variation problem
\begin{equation*}\label{AnisoProblem}
\inf \set{ \PF(\Sigma) \, : \, \Sigma = \partial U, \, |U|=m}, \quad \mbox{with $m>0$ fixed}
\end{equation*}
is homothetic to a closed, convex hypersurface $\W$ called \textit{Wulff shape}. This hypersurface is explicitely defined by the equation
\begin{equation*}\label{gauge}
\W := \set{F^* = 1},
\end{equation*}
where $F^* : \R^{n+1} \mapsto [0,+\infty)$ is the gauge function defined below:
$$F^*(x) := \sup_{v \in \erre^{n+1}} \left\{\langle x, v \rangle : |v|F\left(\frac{v}{|v|}\right) \leq 1\right\}.$$
We recall that the differential of the gauge function satisfies the following property, see \cite[p. 8]{Robin}: 
\begin{equation}\label{Robin}
\restr{dF^*}{z}[c] = \langle \nu_{\W}(z), \, c\rangle , \qquad \forall z \in \W.
\end{equation}

Denoting by $\restr{D^2 F}{x}$ the intrinsic Hessian of $F$ on $\esse^n$ at the point $x$, we define the map $A_F:x \in \esse^n \mapsto \restr{A_F}{x}$ valued in the space of symmetric matrices as follows
\begin{equation*}\label{AF}
\restr{A_F}{x}[z] := \restr{D^2 F}{x}[z] + F(x)z \quad \mbox{ for every } x \in \esse^n, \, z \in T_x \esse^n.
\end{equation*}
Throughout the paper, we will assume that $F$ is an \textit{elliptic integrand}, i.e. $\restr{D^2 F}{x}$ is positive definite at every $x \in \esse^n$. 

For any smooth closed hypersurface $\Sigma$, we can define the anisotropic second fundamental form $S_F$ as
\begin{equation*}\label{AnisoSecondFFDef}
\restr{S_F}{x} \daA{T_x\Sigma}{T_x\Sigma},\ \restr{S_F}{x}:= \restr{A_F}{\nu_\Sigma(x)} \circ \restr{d\nu_\Sigma}{x},
\end{equation*}
and the trace free part of  $S_F$ as
$$\Sdot_F:=S_F-\frac{H_F}{n}g,\qquad \mbox{with} \qquad H_F:=\tr_g (S_F),$$
where $g := \delta_{|\Sigma}$ and $\delta$ is the flat metric on $\R^{n+1}$, see \cite{KoisoPalmer}. 

Then the anisotropic rigidity result proved in \cite[Theorem 1.2]{HeLi} can be stated as follows:
\begin{teo}\label{AnisoSecondFFTheorem}
Let $n \ge 2$ and $\Sigma$ be a closed, oriented hypersurface with $\Sdot_F\equiv 0$, then $\Sigma$ is homothetic to the Wulff shape. 
\end{teo}
Theorem \ref{AnisoSecondFFTheorem} turns out to be stable with respect to the $W^{2,p}$-norm, under the assumption of convexity of the surface $\Sigma$, as proved in \cite[Theorem 1.2]{GiofDeR}. Convexity is deeply used in \cite{GiofDeR}; for instance it directly implies the existence of a parametrization of $\Sigma$ on $\W$. Moreover it is a necessary assumption for general $p$, as showed with the couterexample \cite[Appendix A]{GiofDeR}.  In Section \ref{aniso} of this paper, we show how to drop the convexity assumption in \cite[Theorem 1.2]{GiofDeR} for $p>n$, proving the following:
\begin{teo}\label{MainThm}
Let $n \ge 2$, $\Sigma$ be a closed hypersurface in $\R^{n+1}$ and $p>n$ be given. We assume that there exists $c_0>0$ such that $\Sigma$ satisfies the conditions 
\begin{equation}\label{assumption}
\Vol_n(\Sigma) = \Vol_n(\W), \quad \lVert S_F\rVert_{L^p(\Sigma)} \le c_0.
\end{equation}
There exist $\delta_0$, $C_0>0$ depending only on $n$, $p$, $c_0$ and $\W$ such that, if
\begin{equation}\label{DeltaAnisPinching4Ch}
\lVert \Sdot_F\rVert_{L^p(\Sigma)} \le \delta_0,
\end{equation}
then there exist a parametrization $\psi \fromto{\W}{\Sigma}$ and a vector $c=c(\Sigma)$ satisfying
\begin{equation}\label{MainThmEst}
\lVert \psi - \Id - c\rVert_{W^{2, \, p}(\W)} \le C_0 \lVert \Sdot_F\rVert_{L^p(\Sigma)}.
\end{equation}
\end{teo}
The strategy to remove the convexity assumption in Theorem \ref{MainThm} is the following. First we want to obtain a graph parametrization of $\Sigma$ on $\W$ with small $C^1$-norm, Proposition \ref{QualProp4Ch}. Since a priori $\Sigma$ is not convex, we do not have a priori a parametrization of $\Sigma$ on $\W$. To solve this obstruction, we implement a compactness theorem for immersions in $W^{2,p}$, Proposition \ref{QualCptness4Ch}. The pinching condition \eqref{assumption} is inherited by the limit immersion, thanks to the lower semicontinuity of $\lVert S_F\rVert_{L^p}$ with respect to the $W^{2,p}$-weak convergence, Lemma \ref{AniSemiCont}. Hence, we obtain a map from $\W$ to $\Sigma$ satisfying a $C^1$-bound. However, to produce the desired graph parametrization, we need to show that the projection map from $\Sigma$ onto $\W$ is a diffeomorphism, via a tilt-estimate of the normal. This provides Proposition  \ref{QualProp4Ch}, from which we can conclude the proof of Theorem \ref{MainThm} via a centering argument.

In Section \ref{Isot}, we provide an easier proof of Theorem \ref{MainThm} in the isotropic case $F\equiv 1$, which generalizes the main result in \cite{GiofSec} to not necessarily convex hypersurfaces. In this isotropic setting, the Wulff shape is the round sphere $\S^n$, $S_F$ becomes the classical second fundamental form $h$, and we can give a simpler proof, which does not involve abstract compactness arguments, but simple topological considerations.
\begin{teo}\label{MainCor}
Let  $n \geq 2 $, $\Sigma$ be a closed hypersurface in $\R^{n+1}$ and let $p>n$ be given. We assume that there exists $c_0>0$ such that  $\Sigma$ satisfies the conditions 
\begin{equation}\label{assumptionn}
\Vol_n(\Sigma) = \Vol_n(\S^n), \quad \lVert h\rVert_{L^p(\Sigma)} \le c_0.
\end{equation}
There exist positive numbers $\delta_0$, $C_0>0$ depending only on $n$, $p$, $c_0$ with the following property: if
\[
\lVert \mathring{h}\rVert_{L^p(\Sigma)} \le \delta_0
\]
then there exists a vector $c=c(\Sigma)$ such that $\Sigma - c$ is a graph over the sphere, namely there exists a parametrization
\[
\psi \fromto{\S^n}{\Sigma},\ \psi(x):= e^{f(x)} x,
\] 
and $f$ satisfies the estimate
\begin{equation}\label{MainCorEst}
\lVert f\rVert_{W^{2, \, p}(\S^n)} \le C\lVert \mathring{h}\rVert_{L^p(\Sigma)}.
\end{equation}
\end{teo}

\subsection*{Quasi Einstein hypersurfaces}
The second main result of this paper concerns quantitative stability estimates for quasi Einstein hypersurfaces. In this respect, we need further notation. Let $\Sigma$ be a closed hypersurface in $\erre^{n+1}$ and let us denote with $\Ric$ and $R$ respectively the Ricci tensor and the scalar curvature. We say that $\Sigma$ is an Einstein manifold if the trace-free part of the Ricci tensor 
\begin{equation*}\label{Einstein}
\tRic := \Ric - \frac{1}{n} R g
\end{equation*}
 is identically $0$. In the '$30$s Thomas (see \cite{Thomas}) and Fialkov (see \cite{Fialkow}) independently proved that an Einstein hypersurface $\Sigma$ in $\erre^{n+1}$ with positive scalar curvature is isometric to the round sphere:
\begin{teo}[\cite{Fialkow}, \cite{Thomas}]\label{RigidityRicci2Ch}
Let $\Sigma$ be a closed, connected hypersurface in $\R^{n+1}$ such that 
\[
\tRic=0
\]
at every point. Then $\Sigma$ is a round sphere.
\end{teo} 
The stability properties of this result in the convex setting have been studied in \cite{GiofEin}. The assumption for the validity of the main result in \cite{GiofEin} is the control $0 \le h \le \Lambda g$ on the second fundametal form $h$ of $\Sigma$, which is clearly sub-optimal. Indeed, the bound from below on $h$ implies the convexity of $\Sigma$ (see \cite[Prop. 3.2]{Daniel} for instance), while the bound from above implies \textit{a posteriori} a $W^{2, \, \infty}$ bound on the closeness to the sphere. Since the main result in \cite{GiofEin} provides just a $W^{2, \, p}$ bound, this hypothesis appears abundant. 
One of the aims of this paper is to weaken the assumption on $h$, allowing us to prove in Section \ref{Ricci} the following  theorems:

\begin{teo}\label{Thm1}
Let $n \ge 3$, $\Sigma$ be a closed hypersurface in $\R^{n+1}$ with induced metric $g$ and let $1<p<\infty$ be given. We assume that $\Sigma$ satisfies the conditions 
\begin{equation}\label{PinchingStrict}
\Vol_n(\Sigma) = \Vol_n(\S^n), \quad \Lambda g \le h \mbox{ for some } \Lambda>0.
\end{equation}
There exist $\delta_0$, $C_0>0$ depending only on $n$, $p$, $\Lambda$ with the following property: if
\[
\left \|\tRic \right\|_{L^p(\Sigma)} \le \delta_0
\]
then there exists a parametrization $\psi \fromto{\S^n}{\Sigma}$ and a vector $c=c(\Sigma)$ such that 
\begin{equation}\label{Thm1Est}
\left \|\psi - \Id - c\right \|_{W^{2, \, p}(\S^n)} \le C_0 \left \|\tRic \right \|_{L^p(\Sigma)}.
\end{equation}
\end{teo}

\begin{teo}\label{Thm2}
Let $n \ge 3$, $\Sigma$ be a closed hypersurface in $\R^{n+1}$ and let $p>n$ be given. We assume that there exists $c_0>0$ such that  $\Sigma$ satisfies the conditions 
\begin{equation}\label{Pinchingc0}
\Vol_n(\Sigma) = \Vol_n(\S^n), \quad \|h \|_{L^p(\Sigma)} \le c_0.
\end{equation}
Then for every $q \in (n, \, p)$ there exist $\delta_0$, $C_0>0$ depending only on $n$, $p$, $q$, $c_0$ with the following property: if
\begin{equation}\label{DeltaRicciPinching}
\|\tRic \|_{L^p(\Sigma)} \le \delta_0,
\end{equation}
then there exist a parametrization $\psi \fromto{\S^n}{\Sigma}$ and a vector $c=c(\Sigma)$ such that 
\begin{equation}\label{Thm2Est}
\left \|\psi - \Id - c\right \|_{W^{2, \, q}(\S^n)} \le C_0 \left \|\tRic\right \|_{L^p(\Sigma)}^{\alpha},
\end{equation}
where $\alpha$ is defined as:
\[
\alpha(p, \, q) :=
\begin{cases}
1, &\mbox{ if } n < q \le p/2, \\
p/q - 1, &\mbox{ if } p/2 \le q < p.
\end{cases}
\]
\end{teo} 

In Theorem \ref{Thm1} we remove the assumption on the upper bound on $h$, but we strengthen the convexity, with a uniform bound from below on $h$. In Theorem \ref{Thm2} instead, we completely remove any convexity assumption, obtaining the slightly weaker estimate \eqref{Thm2Est}. We conjecture that the exponent $\alpha$ in the inequality \eqref{Thm2Est} is not optimal. 

The assumptions \eqref{assumption}, \eqref{assumptionn} and \eqref{DeltaRicciPinching}  are often referred to as {\em pinching condition} in the literature.


\section{Notation, preliminaries and strategy of the proof}\label{Not}
\subsection*{Notation}
Throughout the paper, we will use the following notation:

\begin{longtable}{c|c}
$\Vol_n$ & $n$-dimensional Hausdorff measure;\\
$\langle \cdot,\cdot \rangle$ & Euclidean scalar product in $\R^{n+1}$; \\
$\langle \cdot,\cdot \rangle_{L^2}$ & scalar product in $L^2$; \\
$d_{\hd}$ & Hausdorff distance; \\
$\esse^n$ & standard sphere in $\R^{n+1}$; \\
$\W$ & Wulff shape; \\
$\Sigma$ & closed, $n$-dimensional hypersurface in $\R^{n+1}$; \\
$\nu_{\Sigma}$ & outer normal vector field associated to $\Sigma$; \\
$\delta$  & standard metric in $\R^{n+1}$; \\
$\sigma$ & standard metric on $\esse^n$; \\
$g$ & restriction of $\delta$ to $\Sigma$;   \\
$h$ & second fundamental form for $\W$ or $\Sigma$ depending on the context;  \\
$\hdot$ & trace-free part of the second fundamental of $\Sigma$; \\
$H$ & classical mean curvature;  \\
$B^g_r(x)$ & geodesic ball in  $\Sigma$ centred in $x$, of radius $r$; \\
$\Riem$ & Riemann tensor associated to the metric $g$; \\ 
$\Ric$ & Ricci tensor associated to the metric $g$; \\ 
$R$ & scalar curvature associated to the metric $g$; \\
$\B^k_r(x)$ & ball in  $\erre^k$ centred in $x$, of radius $r$ (when $x=0$, we write $\B^k_r$); \\
$\partial$ & usual derivative in $\R^{n+1}$; \\
$D$ & Levi-Civita connection associated to $\esse^n$; \\
$\nabla$ & Levi-Civita connection associated to $\Sigma$ or to $\W$. 
 \end{longtable}

\subsection*{Preliminary results} 
Important tools we will use are the graph parametrizations:

\begin{defi}
Let $\Sigma$ be a closed hypersurface in $\R^{n+1}$, and $q \in \Sigma$ a given point. We say that $\phi_q$ is a \textit{graph parametrization around $q$ with width $R>0$} if $\phi_q$ has the following form:
\begin{equation}\label{GraphChart}
\phi_q \fromto{\B^n_R}{\Sigma},\ \phi_q(z) = q + \Phi_q
\begin{pmatrix}
z \\ 
u_q(z)
\end{pmatrix},
\end{equation} 
where $\Phi_q \fromto{\R^{n+1}}{\R^{n+1}}$ is a matrix in the orthogonal group $O(n+1)$  chosen so that $\Phi_q \cquad{\R^n \times \set{0}} = T_q \Sigma$, $\Phi_q[e_{n+1}]=\nu_\Sigma(q)$.
\end{defi}

Graph parametrizations have great importance in the non-convex case, in view of  Lemma \ref{GraphChartLem}, proved in \cite[Lemma 1.7]{Daniel}, and Remark \ref{noconv}, which justifies the need of the assumption $p>n$: 

\begin{lemma}\label{GraphChartLem}\cite[Lemma 1.7]{Daniel}
Let $n \ge 2$ be given. Let $\Sigma$ be a closed hypersurface in $\R^{n+1}$. Assume there exist $L, \, R>0$ with the following property. For every $q \in \Sigma$ there exists a graph parametrization $\phi_q$ around $q$ with width $R>0$ (as in \eqref{GraphChart}), such that $u_q$ is an L-Lipschitz function.

Then, for every $0 < \rho \le R$, the geodesic ball $B^g_\rho(q)$ satisfies the inclusion 
\begin{equation}\label{GeoBallinclusion}
\phi_q\left(\B^n_{\frac{1}{1+L} \rho}\right) \subset B^g_\rho(q) \subset \phi_q \left(\B^n_\rho\right).
\end{equation}
In particular, for every $q \in \Sigma$ the geodesic ball $B^g_R(q)$ is contained in the chart, and $\Sigma$ can be covered with $N$ such geodesic balls, where $N$ is a natural number depending on $n$, $L$, $R$.
\end{lemma}

\begin{remark}\label{noconv}
As shown in \cite[Section 2.1-2.3]{Daniel}, if $\Sigma$ satisfies \eqref{assumptionn} with $p>n$, then we can find positive constants $L$ and $R$ depending on $c_0$, $n$ and $p$ such that: for every $q \in \Sigma$ there exists a graph parametrization $\phi_q$ around $q$ with width $R>0$, such that $u_q$ is an L-Lipschitz function. In particular Lemma \ref{GraphChartLem} applies. This control can be realized also in the anisotropic case. Indeed condition \eqref{assumption} implies condition \eqref{assumptionn}, as shown in \cite[Proposition 3.3]{GiofDeR}. Although \cite[Proposition 3.3]{GiofDeR} is stated requiring the convexity assumption, one can easily check in its proof that the lower bound estimate 
\begin{equation*}
\|h\|_{L^p(\Sigma)}\leq C \|S_F \|_{L^p(\Sigma)}
\end{equation*}
works also in the non convex setting. The convexity assumption is just needed in the proof of the upper bound $\|S_F \|_{L^p(\Sigma)} \leq C(1+\|\Sdot_F\|_{L^p(\Sigma)})$.
\end{remark}

Throughout all the paper, we shall use only the parametrizations provided by Remark \ref{noconv} by means of \eqref{assumptionn} and will denote them by $\phi_q$. We will use the $\phi_q$ to obtain local estimates and  Lemma \ref{GraphChartLem} to make them global. 
Since we will now work with graph parametrizations, we will use the following lemma, which is stated in \cite[Lemma 1.3]{Daniel}:
\begin{lemma}\label{Computations4Ch}
Let $\phi_q$ be a graph parametrization for $\Sigma$. Then the following formulas hold:
\begin{align}
&g_{ij} = \delta_{ij} + \partial_i u_q \partial_j u_q, \label{gGraph} \\
&g^{ij} = \delta^{ij} - \frac{\partial^i u_q \partial^j u_q}{1 + \abs{\partial u_q}^2}, \label{ginvGraph} \\ 
&\nu_{\Sigma} = \frac{1}{\sqrt{1 + \abs{\partial u_q}^2}} \Phi_q 
\begin{pmatrix}
 \partial u_q \\ 
-1
\end{pmatrix} \label{nuGraph} \\
&h^i_j = \partial_i \coup*{\frac{\partial^j u_q}{\sqrt{1 + \abs{\partial u_q}^2}}} \label{AGraph}
\end{align}
\end{lemma}
The proof of Lemma \ref{Computations4Ch} in \cite{Daniel}  is actually made with the graph parametrisation $\phi(x)=(x, \, u(x))$, i.e. with $q=0$, $\Phi_q = \Id$. However, it can be noted that the action of the isometries  does not change the obtained expressions, because the translations disappear with derivatives and the rotations satisfy $\scal{\Phi[v], \, \Phi[w]} = \scal{v, \, w}$.


To prove the main results, we will need an oscillation estimate. This is given by the following proposition, whose proof is postponed to Appendix \ref{AppendixOsc}

\begin{prop}\label{AnisOscLp3Ch}
Let $n\geq 2$, $n<p< \infty$ and $c_0>0$ be given, and let $\Sigma$ be a closed hypersurface in $\R^{n+1}$ with fixed volume $V$. Let $F$ be an elliptic integrand. 
Assume $\Sigma$ satisfies
$$\norm{h}_{L^p(\Sigma)} \le c_0.$$ 
Then the following estimate is satisfied:
\begin{equation*}\label{AnisOscLpEq3Ch}
\min_{\lambda \in \R} \norm{S_F - \lambda \Id}_{L^p(\Sigma)} \le C(n, \,p, \, c_0, \, F) \norm{\Sdot_F}_{L^p(\Sigma)}.
\end{equation*}
\end{prop}



\section{The isotropic case}\label{Isot}
In this section we prove Theorem \ref{MainCor}. We define for a closed hypersurface $\Sigma$ a radial parametrization to be as follows:
\begin{equation}\label{RadialPar}
\psi \fromto{\S^n}{\Sigma},\ \psi(x):= e^{f(x)} x.
\end{equation}
Moreover we define the barycenter of $\Sigma$ as
\begin{equation*}\label{BarycentreSigma}
b(\Sigma):= \fint_{\Sigma} z \, dV_g(z).
\end{equation*}
The main ingredient for the proof of Theorem \ref{MainCor} is the following proposition:
\begin{prop}\label{MainProp14Ch}
Let  $n \geq 2 $ and $p>n$.
For every $\epsilon>0$ there exists $0<\delta_0=\delta_0(n, \, p, \, c_0, \, \epsilon)$ with the following property. 

Let $\Sigma$ be a closed hypersurface in $\R^{n+1}$ satisfying \eqref{assumptionn}. If 
\[
\norm{\hdot}_{L^p(\Sigma)} \le \delta_0,
\]
then, up to translation, the radial parametrization $\psi \fromto{\S^n}{\Sigma}$ as in \eqref{RadialPar} is well defined, and the logarithmic radius $f$ satisfies
\begin{equation*}\label{RadiusfEst4Ch}
\norm{f}_{C^1(\S^n)} \le \epsilon.
\end{equation*}
\end{prop}

Proposition \ref{MainProp14Ch} is the cornerstone of the section, because it builds the radial parametrization and gives a qualitative estimate of it. 

\subsection*{Proof of Proposition \ref{MainProp14Ch}}
We split the proof of Proposition \ref{MainProp14Ch} in two parts. In the first part we achieve a $C^0$-closeness, in the second part we show how to use this result to build the parametrization. 

We start with a preliminary lemma:
 
\begin{lemma}\label{LemmaGraph}
For every $\epsilon>0$ there exists $0<\delta_0=\delta_0(n, \, p, \, c_0, \, \epsilon)$ with the following property. 

Let $\Sigma$ be a closed hypersurface in $\R^{n+1}$ satisfying \eqref{assumptionn}. If $\norm{h - \lambda_0 g}_{L^p} \le \delta_0$ for some $\lambda_0 \neq 0$, then for every $q \in \Sigma$, for every graph parametrization $\phi_q$ around $q$, we have the following estimate:
\begin{equation*}
\norm*{u_q(\cdot) - \lambda_0^{-1} \coup*{ \sqrt{1 - \lambda_0^2 \abs{ \, \cdot \, }^2} - 1 }}_{C^1} \le \epsilon.
\end{equation*}
\end{lemma}
\begin{proof}
By contradiction, let $(\Sigma^k)_{k \in \N}$ be a sequence of closed hypersurfaces satisfying \eqref{assumptionn} and $\lim_k \norm{h^k - \lambda_0 g^k}_{L^p(\Sigma^k)} = 0$. Let $(q^k)_{k \in \N}$ be a sequence of points $q^k \in \Sigma^k$ such that the associated graph parametrizations satisfy 
\[
\norm*{u^k(\cdot) - \lambda_0^{-1} \coup*{ \sqrt{1 - \lambda_0^2 \abs{ \, \cdot \, }^2} - 1 }}_{C^1} \ge \epsilon_0>0.
\]
We show how this is not possible, using an idea of \cite[Cor. 1.2]{Daniel}. Firstly, we can assume w.l.o.g. that every $q^k$ is equal to $\lambda_0^{-1} e_{n+1}$ and $\Phi_{q^k}=\Id$. Since every $\Sigma^k$ satisfies \eqref{assumptionn}, by Remark \ref{noconv} we consider the graph parametrizations $\phi^k$ associated to $q^k$. The properties of $\phi^k$ combined with \eqref{assumptionn} grant us:
\[
\sup_k \norm{u^k}_{W^{2, \, p}(\B^n_R)} \le c(n, \,p, \, c_0)<+\infty.
\]
Let us set $v^k:= \frac{\partial u^k}{\sqrt{1 + \abs{\partial u^k}^2}}$. Then, from \eqref{AGraph} and the contradiction hypothesis, we obtain 
\[
\lim_k \norm{\partial v^k - \lambda_0 \Id}_{L^p(\B_R)} = \lim_k \norm{\partial \coup*{v^k - \lambda_0 x}}_{L^p(\B_R)}=0.
\]
Setting $c^k = \fint v^k$, we get from Sobolev inequalities
\[
\lim_k \norm{v^k - c^k - \lambda_0 x}_{W^{1, \, p}(\B^n_R)} = 0.
\]
Now, since $c^k$ is clearly bounded and $v^k(0)=0$ for every $k\in \N$ and $n<p$, we also obtain the convergence 
\[
\lim_k \norm{v^k - \lambda_0 x}_{W^{1, \, p}(\B^n_R)} = 0.
\]
Let us define the function 
\[
\xi \fromto{\B^n_1}{\R^n},\, \xi(x):= \frac{x}{\sqrt{1 - \abs{x}^2}}.
\]
The function $\xi$ is smooth and has bounded derivatives in the ball $\B^n_{\rho}$ with $\rho\le \frac{1}{2}$. Moreover it satisfies the equality 
\[
\xi(v^k) = \frac{1}{\sqrt{1 + \abs{\partial u^k}^2}} \frac{\partial u^k}{\sqrt{1 - \abs{\partial u^k}^2/\coup*{1 + \abs{\partial u^k}^2}}} = \partial u^k.
\]
We obtain: 
\begin{align*}
\lim_k \norm{\xi(v^k) - \xi(\lambda_0 x)}_{W^{1, \, p}(\B^n_R)} 
&= \norm*{\partial u^k - \frac{\lambda_0 x}{\sqrt{1 - \lambda_0^2 \abs{x}^2}}}_{W^{1, \, p}(\B^n_R)} \\ 
&= \lim_k \norm*{\partial \coup*{ u^k - \lambda_0^{-1}\sqrt{1 - \lambda_0^2 \abs{x}^2} }}_{W^{1, \, p}(\B^n_R)}=0.
\end{align*}
With the same argument as before, we observe that $u^k$ is converging in $W^{2, \, p}$ to  
$\lambda_0^{-1}\sqrt{1 - \lambda_0^2 \abs{x}^2}$, and this is the desired contradiction.
\end{proof}

\begin{remark}
We remark that in Lemma \ref{LemmaGraph} we do not claim that $\lambda_0$ has to be equal to $1$. The problem of finding the ``right'' $\lambda_0$ will be solved in the second part, when we will build the parametrization. The requirement of $\lambda_0$ being not $0$ is instead necessary, however as shown in \cite[Remark 1.9]{Daniel} a closed hypersurface $\Sigma$ must satisfy the lower bound
\begin{equation}\label{lambda0nonzero}
\norm{h}_{L^p(\Sigma)} \ge C(n, \, p, \, \Vol_n(\Sigma)).
\end{equation}
Since in our case $\Vol_n(\Sigma)=\Vol_n(\S^n)$, we avoid such degenerate cases. 
\end{remark}

Next we show how Lemma \ref{LemmaGraph} leads to a $C^1$-closeness to the sphere. 
\begin{cor}\label{CorClose}
For every $\epsilon>0$ there exists $0<\delta_0=\delta_0(n, \, p, \, c_0, \, \epsilon)$ such that, under the hypothesis of Lemma \ref{LemmaGraph}, 
\begin{equation}\label{C01Closenesss}
d_{\hd}\coup*{\Sigma, \, \S^n_{\abs{\lambda_0}^{-1}}}\leq \epsilon, \qquad \mbox{ and } \qquad \abs*{T_q \Sigma - \scal{q}^\perp} \le  \epsilon \quad \forall q\in \Sigma.
\end{equation}
\end{cor}

\begin{proof}
Let $\Sigma$, $0<\epsilon$ and $0<\delta_0$ be given as in Lemma \ref{LemmaGraph}. We choose a point $q \in \Sigma$, then rotate and translate $\Sigma$ so that $q=-\lambda_0^{-1} e_{n+1}$, $T_{q} \Sigma = \R^n \times \set{0}$. Hence the parametrization has the simpler form $\phi_q(x) = -\lambda_0^{-1} e_{n+1} + (x, \, u_q(x))$ and parametrizes a portion of the sphere $\S^n_{\abs{\lambda_0}^{-1}}$. For every $\tilde q \in \phi_q(\B^n_R)$, we consider $\tilde z \in \B^n_R$ such that $\tilde q=\phi_q(\tilde z)$. Then the following inequalities easily hold:
\[
\abs*{\tilde q - \coup*{\tilde z, \, \lambda_0^{-1}\sqrt{1 - \lambda_0^2 \abs{\tilde z}^2} }} \le \epsilon, \quad
\abs*{T_{\tilde q} \Sigma - \left \langle \coup*{\tilde z, \, \lambda_0^{-1} \sqrt{1 - \lambda_0^2 \abs*{\tilde z}^2 } } \right \rangle^\perp } \le \epsilon
\]
Now we apply Lemma \ref{GraphChartLem}: For every parametrisation $\phi_q$ we can find a geodesic ball $B^g_{\rho}(q)$ with $\rho=\rho(n, \, p, \, c_0)$ and satisfying condition \eqref{GeoBallinclusion}, namely
\[
\phi_q\coup*{ \B^n_{\frac{1}{1+L} \rho} } \subset B^g_\rho(q) \subset \phi_q\coup*{\B^n_\rho}.
\]
Via Lemma \ref{GraphChartLem} we can easily obtain a covering of $N$ geodesic balls $B^g(q_1), \, \dots B^g(q_N)$, where $N \le N_0(n, \, p, \, c_0)$ such that Lemma \ref{LemmaGraph} holds for $\phi_{q_1}, \, \dots \phi_{q_N}$. By a simple induction we easily find a constant $c(n, \,p, \, c_0)$ such that for every $q \in \Sigma$
\begin{equation}\label{C01Closeness}
\abs*{q - \abs{\lambda_0}^{-1} \frac{q}{\abs{q}}} \le c \epsilon, \quad \abs*{T_q \Sigma - \scal{q}^\perp} \le c \epsilon.
\end{equation}
This proves the $C^0$-closeness.
\end{proof}

We finish the proof of Proposition \ref{MainProp14Ch} by proving that $\Sigma$ can be parametrized on the sphere as in \eqref{RadialPar} and that $\lambda_0=1$. 
We define the projection
\[
p \fromto{\Sigma}{\S^n_{\abs{\lambda_0}^{-1} }},\ p(q):= \abs{\lambda_0}^{-1} \frac{q}{\abs{q}}.
\]
We start by proving that $p$ is a local diffeomorphism. The map is clearly differentiable, and a straight computation proves that the differential of $p$ at $q \in \Sigma$ is given by 
\begin{equation*}\label{ProjDiff}
\restr{dp}{q} \fromto{T_x \Sigma}{T_{p(q)} \S^n},\ \restr{dp}{q}[v]=  
\frac{\abs{\lambda_0}^{-1} }{|q|} \left( v - \left\langle v, \, \frac{q}{|q|}\right\rangle \frac{q}{|q|}  \right).
\end{equation*}
It is easy to see that $\ker \restr{dp}{q} =  \Set{tq | t \in \R}$. We want to prove that the differential $\restr{dp}{q}$ has maximal rank at every $q$, and this will prove that $p$ is a local diffeomorphism. In order to achieve this goal, we just need to show that for every $q \in \Sigma$, $q$ does not belong to $T_q \Sigma$, and this is exactly what \eqref{C01Closenesss} implies.
Hence $p$ is a local diffeomorphism. Let us show that it is a global one. 
Indeed, we consider the multiplicity function 
\[
\eta \fromto{\S^n}{\N},\ \eta(x):= \sum_{p(q)=x} 1.
\]
The function $\eta$ is well-defined, and since $p$ is a local diffeomorphism, it is continuous, thus  necessarily constant, say $\eta \equiv Q$. Then it is a $Q$-covering, but since $\S^n$ is simply connected, we must have $Q=1$, and hence $p$ is a diffeomorphism.
Let us define $\psi:=p^{-1}$. By construction, we find that $\psi(x)=e^{f(x)}x$ as in \eqref{RadialPar}, and \eqref{C01Closenesss} tells us that $f$ has small $C^1$-norm. This concludes the construction.

Finally we can conclude the proof of the proposition. Let us argue by compactness and consider a sequence of closed hypersurfaces $(\Sigma^k)_{k \in \N}$ satisfying \eqref{assumptionn}, and $\|\hdot^k\|_{L^p(\Sigma^k)} \to 0$, where $\hdot^k$ is the trace-free part of the second fundamental form of $\Sigma^k$. From Proposition \eqref{AnisOscLp3Ch} applied with $F\equiv 1$, we are able to find a sequence $(\lambda^k)_{k \in \N}$ such that, for every $k \in \N$,
\begin{equation*}\label{MeanOscillationLambdak}
\lVert h^k - \lambda^k \Id\rVert_{L^p(\Sigma^k)} \le C(n, \, p, \, c_0) \|\hdot^k\|_{L^p(\Sigma^k)} \to 0.
\end{equation*}
The sequence $(\lambda^k)_{k \in \N}$ is clearly bounded. Up to extraction of a subsequence we can assume $\lambda^k \to \lambda_0$ which has to be non-zero because of \eqref{lambda0nonzero}. 
We show the equality $\abs{\lambda_0}=1$ as an application of the area formula. Indeed, combining Lemma \ref{LemmaGraph} and Corollary \ref{CorClose} we obtain that, up to translations, the hypersurfaces $\Sigma^k$ are radially parametrized by a map
\[
\psi^k \fromto{\S^n_{ \abs{\lambda_0}^{-1} }}{\Sigma},\,  \psi(x) = e^{f^k(x)} x,\qquad \mbox{with} \qquad \norm{f}_{C^1} \le \epsilon.
\]
Then, we have:
\begin{align*}
1= \frac{\Vol_n(\Sigma)}{\Vol_n(\S^n)} &= \abs{\lambda_0}^{-1} \fint_{\S^n_{\abs{\lambda_0}^{-1} }} e^{nf^k} \sqrt{1 + \abs{\nabla f^k}^2} \, dV_\sigma = \abs{\lambda_0}^{-1}  \coup*{1 + O\coup*{\norm{f^k}_{C^1}}}.
\end{align*}
For $k \to \infty$ we obtain that $\abs{\lambda_0}=1$. The conclusion of the proposition follows by showing that $\lambda_0=1$, thus implying that every subsequence of $(\lambda^k)_{k \in \N}$ converges to $1$ and hence the whole sequence. Firstly, we notice that every $n \lambda^k$ must be close to the average of the mean curvature $\overline{H^k}$. Indeed, 
\[
\abs{\overline{H^k} - n \lambda^k} \le \fint_{\Sigma^k} \abs*{H^k - n \lambda^k} 
= \fint_{\Sigma^k} \abs*{ \scal{h^k - \lambda^k g^k, \, g^k}} \le C(n, \,p, \, c_0) \norm{\hdot^k}_{L^p(\Sigma^k)} \downarrow 0.
\]
Now we show that $\overline{H^k}$ must be close to $n$ and conclude. This follows by a simple estimate.
\begin{align*}
\overline{H^k}
&= \fint_{\S^n} ne^{(n-1)f^k} - \fint_{\S^n} \divv \coup*{\frac{\nabla f^k}{\sqrt{1 + \abs{\nabla f^k}^2}}} e^{(n-1)f^k} \sqrt{1 + \abs*{\nabla f^k}^2}  \\
&= n \fint_{\S^n} n e^{(n-1)f^k} + \fint_{\S^n} e^{(n-1)f^k} \coup*{\frac{(n-1)\abs{\nabla f^k}^2}{\sqrt{1 + \abs{\nabla f^k}^2}} + \frac{\nabla^2 f^k[\nabla f^k, \, \nabla f^k]}{1 + \abs{\nabla f^k}^2}   }
\end{align*}
Since every $\Sigma^k$ satisfies \eqref{assumptionn}, we easily obtain that the sequence $(f^k)_{k \in \N}$ is uniformly $W^{2, \, p}$-bounded, and thus
\[
\abs*{\overline{H^k} - n} \le C(n, \, p, \, c_0) \norm{f^k}_{C^1} \downarrow 0. 
\]
This shows that $\lambda_0$ must be equal to $1$, and all the computations we have made do not actually depend on the chosen subsequence.

\subsection*{Conclusion}

Insofar we have found a qualitative convergence. We will now make it quantitative. 
%

Indeed, with the very same proof  of \cite[Proposition 2.3]{GiofSec}, we are able to show the following result. Notice that the convexity assumption in \cite[Proposition 2.3]{GiofSec} is actually never used in its proof.
\begin{prop}\label{AlmostMainThmIsot}
Let $\Sigma$ be a closed hypersurface in $\erre^{n+1}$ satisfying \eqref{assumptionn}. Then for every $\epsilon>0$ there exists $0<\delta=\delta(\epsilon, \, n, \, p, \, c_0)$ wit the following property:
 if 
\[
\|\hdot\|_{L^p} \le \delta,
\]
then $\Sigma$ admits a radial parametrization and its radius $f$ satisfies the following inequality
 \begin{equation*}\label{MainEstimateEqsfera}
 \lVert f - \phi_f\rVert_{W^{2, \, p}(\S^n)} \le C \left( \|\hdot\|_{L^p(\Sigma)} + \sqrt{\epsilon} \|f\|_{W^{2, \, p}(\S^n)} \right),
 \end{equation*}
 where $C=C(n, \, p, \, c_0)$, and we have denoted
 \begin{equation*}\label{VecDefIsot}
\phi_f(z):= \langle z,v_f \rangle, \qquad \text{where} \qquad  v_f:= \frac{1}{n+1} \fint_{\S^n} z f(z) \, dV.
 \end{equation*}
\end{prop}
In order to prove Theorem \ref{MainCor}, we just have to show that we can center $\Sigma$ so that $v_f=0$. This can be done by proving that we can center the hypersurface so that $b(\Sigma)=0$.  

Since $\Sigma$ is not convex, it can be a priori impossible to translate it and keep a radial parametrization. However, this is not a problem, and it is done by looking carefully at the proof of Lemma \ref{LemmaGraph}. In the proof of Lemma \ref{LemmaGraph} we chose a random point $q^k \in \Sigma^k$ and fix it to be $-\lambda_0 e_{n+1}$, then we perform our analysis. In order to center $\Sigma$ better, we just improve the proof in Lemma \ref{LemmaGraph} by choosing better $q^k$. 
Indeed, let again $(\Sigma^k)_{k \in \N}$ be a sequence of hypersurfaces satisfying \eqref{assumptionn} and $\lim_k \norm{\hdot^k}_{L^p(\Sigma^k)}=0$. We apply a translation so that $b(\Sigma^k)=0$ for every $k$, and choose $q^k$ so that 
\[
\abs{q^k}^2 = \max_{q \in \Sigma} \abs{q}^2. 
\]
It is easy to see that for such choice we have the equality $T_{q^k} \Sigma^k = \scal{q^k}^\perp$. The study we made above also grants us the limit: 
\[
\lim_k \norm{h^k - \Id}_{L^p(\Sigma^k)} = 0.
\]
We follow again the same argument of Lemma \ref{LemmaGraph}, choosing this time $q^k$ as first point for the covering argument, and obtain that the sequence $(\Sigma^k)_{k \in \N}$ is converging to a sphere $\S^n(c)$ with center $c$. Since the barycenter condition $b(\Sigma^k)=0$ passes to the limit, we also obtain that this sphere must satisfy $b(\S^n(c))=0$, therefore implying $c=0$. Now we repeat the same argument, and obtain the following improved version of Proposition \ref{MainProp14Ch}:
\begin{prop}
For every $0<\epsilon$ there exists $0<\delta_0=\delta_0(n, \, p, \, c_0, \, \epsilon)$ with the following property. 

If $\Sigma$ is a closed hypersurface satisfying \eqref{assumptionn} and $\norm{\hdot}_{L^p(\Sigma)} \le \delta_0$, then there exists a vector $c \in \R^{n+1}$ such that $b(\Sigma-c)=0$ and the radial parametrization
\[
\psi \fromto{\S^n}{\Sigma-c},\ \psi(x)=e^{f(x)} x
\]
is well defined. Moreover, $\norm{f}_{C^1(\Sigma)} \le \epsilon$.
\end{prop}
Via this proposition and the discussion above, we obtain Theorem \ref{MainCor}.

\section{The anisotropic case}\label{aniso}
We define the parametrization $\psi$ we will use in the proof of Theorem \ref{MainThm}. 
Let $\Sigma$ be a closed  hypersurface in $\R^{n+1}$ which is contained in the tubular neighborhood $B_\epsilon(\W)$ associated to $\W$, that is the set
\begin{equation*}\label{Tubular}
B_\epsilon(\W):= \set{ z \in \R^{n+1} \mid z= x+ \rho \nu_{\W}(x), \quad \forall  x \in \W, \, 0 \le \rho <\epsilon   }.
\end{equation*}
We refer the reader to \cite[Chapter 5]{Hirsch} for the proof of the following properties on the tubular neighborhoods. We recall that there exits $\epsilon>0$ sufficiently small such that, for every $r<\varepsilon$, $B_r(\W)$ is an open, bounded set with smooth boundary diffeomorphic to $\W$. 
We say that $\Sigma$ admits a \textit{radial parametrization} if there exists a diffeomorphism
\begin{equation}\label{RadialPar1}
\psi \daA{\W}{\Sigma},\ \psi(x)= x + u(x) \nu_{\W}(x),  \qquad \mbox{ for some } u \in C^\infty(\W).
\end{equation}
The function $u$ shall be called the \textit{radius} of $\Sigma$. 

In order to further exploit radial parametrizations, we need the following notation:
\begin{defi}\label{deff}
For every $c \in \R^{n+1}$, we define
 \begin{equation}\label{PhiDef}
 \phi_c \daA{\W}{\R},\quad \phi_c(y) := \langle c, \, \nu_{\W}(y)\rangle.
 \end{equation}
 We fix the vectors $\set{w_i}_{i=1}^{n+1}\subset \R^{n+1}$ such that  the associated functions  $\phi_{w_i}$ are an orthonormal frame  in $L^2$ for the vector space $\set{\phi_c}_{c \in \erre^{n+1}}$.  
 For every function $u \daA{\W}{\R}$, we define the vector $v_u \in \R^{n+1}$
\begin{equation*}\label{MeanMainC}
v_u  := \sum_{i=1}^{n+1} \langle u, \, \phi_{w_i}\rangle_{L^2} w_i.
\end{equation*}
\end{defi}

A useful tool about radial parametrizations is the following theorem (see \cite[Theorem 5.1]{GiofDeR}), which allows us to pass from a qualitative closedness to a quantitative one:
\begin{teo}\label{MainEstimateb}
Let $n \ge 2$, $p>n$ and let  $\Sigma$ be a closed, radially parametrized hypersurface in $\R^{n+1}$, satisfying \eqref{assumption} and having radius $u$ verifying
 \[
\lVert u\rVert_{C^0} \le \epsilon, \quad \lVert \nabla u \rVert_{C^0} \le C(n,F)\sqrt{\epsilon} .
\] Then there exists a constant $C=C(n, \, p, \, F)>0$ such that 
 \begin{equation*}\label{MainEstimateEq}
 \inf_{c \in \erre^{n+1}} \lVert u - \phi_c\rVert_{W^{2, \, p}(\W)} \le C \left( \lVert \Sdot_F(\Sigma)\rVert_{L^p(\Sigma)} + \sqrt{\epsilon} \lVert u\rVert_{W^{2, \, p}(\W)} \right).
 \end{equation*}
 \end{teo}
Although it is stated in \cite[Theorem 5.1]{GiofDeR} under the convexity hypothesis, the proof makes no other use of it rather than allowing Proposition \ref{AnisOscLp3Ch} (Theorem 3.1 in \cite{GiofDeR}). As we discussed in Remark \ref{noconv}, we can just replace the convexity assumption with the hypothesis \eqref{assumption} when $p>n$.
The use of the linear projections $\phi_c$ defined in \eqref{PhiDef} appears natural in view of \cite[Theorem 5.4]{GiofDeR}, which characterizes these functions as the only elements of the kernel of the anisotropic stability operator.

The cornerstone of the proof is the following proposition: 

\begin{prop}\label{QualCptness4Ch}
Let $n \in \N$, $n<p$, $0<\mathcal{A}, \, \mathcal{V}, \, R$ positive constants. Let $\mathfrak{F}$ be the set of all couples $\coup*{M, \, f}$ with the following properties:
\begin{itemize}
\item $M$ is an $n$-dimensional, compact manifold (without boundary).
\item $f \in W^{2, \, p}(M, \, \R^n)$ is an immersion with
$$
\norm{h(f)}_{L^p(M)} \le \mathcal{A},  \qquad \Vol_n(M) \le \mathcal{V}, \qquad
f(M) \subset \B^n_R.
$$
\end{itemize}
Then for every sequence $f_i \fromto{M_i}{\R^n}$ in $\mathfrak{F}$ there exist a subsequence $f_j$ , a mapping $f \fromto{M}{\R^n}$ in $\mathfrak{F}$, and a sequence of diffeomorphisms $\phi \fromto{M}{M_j}$, such that $f_j \circ \phi_j$ converges weakly in $W^{2, \, p}(M, \, \R^n)$ to $f$.
\end{prop}
Proposition \ref{QualCptness4Ch} is part of a series of compactness theorems on immersions started in \cite{Langer}, where the author proves the result for immersed surfaces in $\R^3$, and then continued in \cite{Delladio} for immersed hypersurfaces, and in \cite{Breuning} for the general case.
The proposition we want to prove is the following:

\begin{prop}\label{QualProp4Ch}
Let $\Sigma$ be a closed hypersurface in $\R^{n+1}$ satisfying \eqref{assumption}.
For every $0<\epsilon$ sufficiently small there exists a $0<\delta=\delta(\epsilon, \, n, \, p, \, c_0, \, \W)$ with the following property. If $\Sigma$ satisfies \eqref{DeltaAnisPinching4Ch}, then it admits an anisotropic radial parametrization as in \eqref{RadialPar1}. Moreover the radius $u$ satisfies the estimate
\begin{equation} \label{QualThmEst}
\norm{u}_{C^1} \le \epsilon.
\end{equation}
\end{prop}

We will see in the conclusion how the qualitative $C^1$-closeness will bring the desired quantitative one. 

\subsection*{Proof of Proposition \ref{QualProp4Ch}}
The proof of Proposition \ref{QualProp4Ch} uses strongly the compactness result of Proposition \ref{QualCptness4Ch}. Firstly, we prove the following two lemmas.

\begin{lemma}\label{AniSemiCont}
Let $\phi_k \fromto{M}{\R^{n+1}}$ be a sequence of immersions of a closed manifold. Assume $\phi_k$ satisfies \eqref{assumption}, and $\phi_k$ converges to an immersion $\phi_0$, weakly in $W^{2, \, p}$ . Then we have the inequality
\begin{equation*}\label{AniSemiContEq}
\norm*{S_F(\phi_0)}_{L^p(M)} \le \liminf_k \norm{S_F(\phi_k)}_{L^p(M)}.
\end{equation*}
\end{lemma}

\begin{lemma}\label{ConvWulff4Ch}
Let $(\Sigma^k)_{k \in \N}$ be a sequence of hypersurfaces satisfying \eqref{assumption}, and such that we have also ${\lim_k \norm{\Sdot_F^k}_{L^p(\Sigma^k)} = 0}$. Then there exist a non relabeled subsequence $(\Sigma^k)_{k \in \N}$ and parametrizations $\eta_k \fromto{\W}{\Sigma^k}$ such that $\eta_k$ converges weakly in $W^{2, \, p}$ to the identity map $\Id \fromto{\W}{\W}$.
\end{lemma}

Let us prove the lemmas and then show how they imply Proposition \ref{QualProp4Ch}.

\begin{proof}[Proof of Lemma \ref{AniSemiCont}]
We introduce the map
\begin{equation*}\label{WulffPar}
\Psi \fromto{\S^n}{\R^{n+1}},\ \Psi(x) := \grad_\sigma F(x) + F(x)x.
\end{equation*}
From \cite{Palmer} we know that the map $\Psi$ parametrizes the Wulff shape. It is immediate to show the equality
\begin{equation}\label{Diffeq4Ch}
S_F := A_F \circ d \nu = d\coup*{\Psi \circ \nu}.
\end{equation}
Indeed, the differential of $\Psi$ has the following form:
\[
d\Psi \cquad*{\frac{\partial}{\partial \theta^i}} = 
\underbrace{ \frac{\partial}{\partial \theta^i} \coup*{\grad_\sigma F} + \partial_i F \, \Id}_{ = D_i \coup*{D F} } + F \frac{\partial}{\partial \theta^i} = \coup*{A_F}^j_i \frac{\partial}{\partial \theta^j},
\]
where we have denoted by $D$ the Levi-Civita connection compatible with the canonical metric on the round sphere.
Taking the composition we obtain \eqref{Diffeq4Ch}. Let now $\coup*{\nu_k}_{k \in \N}$ be the sequence of outer normals associated to $\phi_k$, i.e. the sequence of mappings $\nu_k \fromto{M}{\S^n}$ such that 
\[
\scal{\nu_k(q), \, \restr{d \phi_k}{q} [v]} = 0, \, \forall \, v \in T_q M,
\]
and with orientation fixed so every $\nu_k$ is the outer normal for $\phi_k(M)=\Sigma^k$.
We claim that the sequence $\coup*{\nu_k}_{k \in \N}$ is bounded in $W^{1, \, p}(M, \, \R^{n+1})$. Firstly,  since $S_F^k = A_F \circ h^k$, we obtain
\begin{equation*}
\norm{h^k}_{L^p(\Sigma^k)} = \norm*{\coup*{A_F}^{-1} S_F^k }_{L^p(\Sigma^k)} \le c(F) \, c_0 = C(F, \, c_0),
\end{equation*}
and thus \eqref{assumption} implies \eqref{assumptionn}. Now we show how the $L^p$-boundedness of the second fundamental forms gives us the $L^p$-boundedness of the differential of the normals. The key is the following proposition, proved in \cite[Thm. 6.3]{Delladio}.

\begin{prop}\label{Delladio2ndff}
Let $2 \le p$, and $\psi \fromto{\B^n_R}{\R^{n+1}}$, $\psi(x)=\coup*{x, \, \xi(x)}$ be a graph parametrisation, with $\xi$ smooth function. Then the following estimate holds: 
\begin{equation}\label{Estimate2ndff}
\norm{\partial^2 \xi}_{L^p(\B^n_R)} \le \coup*{1 + \norm{\partial \xi}_0}^{\frac{3p-1}{p}} \norm{h}_{L^p}.
\end{equation}
\end{prop}
Estimate \eqref{Estimate2ndff} allows us to conclude. Indeed, since our hypersurfaces satisfy \eqref{assumption}, then by Remark \ref{noconv} we can apply Lemma \ref{GraphChartLem}. Plugging \eqref{Estimate2ndff} we can easily find a radius $R$ depending on $n$, $p$, $c_0$ such that the estimate
\[
\norm{d\nu_k}_{L^p(B^{g_k}_R(q))} \le C(n, \, p, \, c_0) \norm{h^k}_{L^p} \le C(n, \, p, \, c_0, \, F) \norm{S_F^k}_{L^p}
\]
holds for every point $q\in M$. Then we make this estimate global via Lemma \ref{GraphChartLem}, and obtain:
\[
\norm{d\nu_k}_{L^p} \le C(n, \, p, \, c_0, \, \W)\norm{S_F^k}_{L^p}.
\]
Therefore our sequence $\coup*{\nu_k}_{k \in \N}$ is bounded in $W^{1, \, p}$. Since $n<p$, every weak $W^{1, \, p}$-limit point $\nu_0$ is also a strong $C^{0, \, \alpha}$-limit point, and satisfies
\[
\abs*{\nu_0(q)}=1, \quad \scal{\nu_0(q), \, \restr{d\phi_0}{q}[v]} = 0, \qquad \mbox{for every $q\in M$ and $v \in T_q M$}.
\]
This shows that $\nu_0$ is the outer normal associated to the immersion $\phi_0$, and moreover $d\nu_k \debto d\nu_0$. In order to complete the proof, we simply consider equality \eqref{Diffeq4Ch}: since the map $\Psi$ is smooth, we obtain that $\Psi \circ \nu_k$ converges to $\Psi \circ \nu_0$ weakly in $W^{1, \, p}$, and the result follows from classical Sobolev theory.
\end{proof}

With the help of Lemma \ref{AniSemiCont} we prove Lemma \ref{ConvWulff4Ch}. 
\begin{proof}[Proof of Lemma \ref{ConvWulff4Ch}]
Let us argue by contradiction, and assume there exists a sequence of closed hypersurfaces $(\Sigma^k)_{k \in \N}$ satisfying \eqref{assumption}, $\lim_k \norm{\Sdot_F^k}_{L^p(\Sigma^k)}=0$, all enclosed in a ball $\B^{n+1}_R$, and such that the conclusion of the proposition does not hold.

We apply Proposition \ref{QualCptness4Ch}, and find a non-relabeled subsequence $(\Sigma^k)_{k \in \N}$, a closed manifold $M$, parametrizations $\phi_k \fromto{M}{\Sigma^k}$ converging weakly in $W^{2, \, p}$ to an immersion $\phi_0$. From \eqref{assumption}, Remark \ref{noconv} and Proposition \ref{AnisOscLp3Ch}, we find the existence of a bounded sequence $(\lambda^k)_{k \in \N}$ such that 
\[
\norm{S_F^k - \lambda^k \Id}_{L^p(\Sigma^k)} \le C\norm{\Sdot_F^k}_{(\Sigma^k)} \downarrow 0.
\]
Up to subsequences, we assume $\lambda^k = \lambda_0$ for every $k\in \N$. Then $\lambda_0$ must be different from $0$ because of the estimate
\[
\norm{S_F}_{L^p(\Sigma)} \ge C(n, \, p, \, F) \norm{h}_{L^p(\Sigma)} \ge C(n, \, p, \, F, \, \Vol_n(\Sigma)) = C(n, \, p, \, F).
\]
Since $S_F= d\coup*{\Psi \circ \nu_h}$, we apply Lemma \ref{AniSemiCont} to the sequence $\Psi \circ \nu_h - \lambda_0 x$, and obtain that the limit immersion $\phi_0$ satisfies the equality 
\[
S_F = \lambda_0 \Id
\]
weakly. From it we easily infer
\begin{equation}\label{fastweak}
h(\phi_0) = \lambda_0 \coup*{A_F}^{-1}.
\end{equation}
Now we take the trace in \eqref{fastweak}, and obtain that in every graph parametrisation around every point $q\in \Sigma$, the function $u_q$ that parametrizes the immersion is Lipschitz and satisfies an equality of the following type:
\[
\divv \coup*{ \frac{\partial u_q}{\sqrt{1 + \abs{\partial u_q}^2}} } = f(u_q, \, \partial u_q),
\]
for a certain smooth function $f$. This tells us that the function $u_q$ is smooth. Since then \eqref{fastweak} holds classically, $u_q$ is also convex, and we obtain that $\phi_0$ is a smooth immersion and $\Sigma_0 := \phi_0(M)$ is a smooth, convex hypersurface of $\R^{n+1}$. Since $\Sigma_0$ is diffeomorphic to a round sphere, the same argument used to build the parametrization in the proof of \ref{MainProp14Ch} tells us that $\phi_0$ is actually an embedding. From \cite{Palmer} and the volume condition in \eqref{assumption}, we conclude that $\lambda_0=1$ and $\phi_0(M)$ must be a Wulff shape $\W+c$ for some vector $c \in \R^{n+1}$. Up to translation, we assume $c=0$. Now we easily define $\eta^k \fromto{\W}{}\Sigma^k$, $\eta^k=\phi_k \circ \phi_0^{-1}$ and obtain that $\eta^k$ converges to the identity map $\Id \fromto{\W}{\W}$ weakly in $W^{2, \, p}$.
\end{proof}
The results obtained give us a priori only a qualitative $C^1$-closeness. 
Insofar we have proved the following result:
\begin{cor}\label{CorollaryWC1}
Let $\Sigma \subset \R^{n+1}$ be a closed hypersurface satisfying conditions \eqref{assumption}.
Then for every $\epsilon>0$ there exists $\delta_0(n, \, p, \, \W, \, c_0)>0$ with the following property. If  $\Sigma$ satisfies \eqref{DeltaAnisPinching4Ch}, then there exists a map $\eta \fromto{\W}{\Sigma}$ such that 
\begin{equation}\label{etaC1estimate}
\norm{\eta - \Id}_{C^{1, \, \alpha}(\W)} \le \epsilon.
\end{equation}
\end{cor}

We show how \eqref{etaC1estimate} yields the desired graph parametrization. Let $\Sigma$ be a closed hypersurface that satisfies the assumptions of Corollary \ref{CorollaryWC1}. Let $B_\epsilon(\W)$ be the tubular neighbourhood associated to $\W$. We denote by $P$ the natural projection over the Wulff shape, that is 
\[
P \fromto{B_\epsilon(\W)}{\W}, \quad P \colon q=x+\rho \nu_\W(x) \longrightarrow x.
\]
The map $P$ is Lipschitz and smooth. Moreover, it can be proved that for every $q \in B_\epsilon \W$, the differential $\restr{dP}{q} \fromto{\R^{n+1}}{T_{P(q)} \W }$ is surjective and satisfies the property 
\begin{equation}\label{KernelP}
\restr{dP}{q}[z]=0 \, \Leftrightarrow \, z= \lambda \nu_{\W} \coup*{P(q)},\ \lambda \in \R.
\end{equation} 
See \cite[Ch. 5]{Hirsch} for the details. Since $\Sigma$ satisfies Corollary \ref{CorollaryWC1} and hence estimate \eqref{etaC1estimate}, then $\Sigma \subset B_\epsilon(\W)$ and we can set $p:= \restr{P}{\Sigma}$. Therefore, $p$ is a smooth, Lipschitz map from $\Sigma$ to $\W$ and satisfies 
\begin{equation}\label{pC0}
\sup_{q \in \Sigma} \, \abs{q - p(q)} \le \epsilon.
\end{equation}
We claim that $p$ also satisfies:
\begin{equation}\label{closenormals}
\sup_{q \in \Sigma} \, \abs{\nu_\Sigma(q) - \nu_\W (p(q))}\le C(n, \, \W) \epsilon.
\end{equation}  
If the claim is true, then $p$ is a local diffeomorphism: indeed, since $\nu_\W(p(q)) \notin T_q \Sigma$ for every $q \in \Sigma$, by \eqref{KernelP} $\restr{dp}{q}$ has maximal rank at every point $q \in \Sigma$. Hence $p$ is a local diffeomorphism, and since the Wulff shape is diffeomorphic to the sphere, the same argument made in the isotropic case proves it is a global diffeomorphism. Then the inverse $\psi(x)= x + u(x) \nu_\W(x)$ is the desired radial parametrization and from inequalities \eqref{pC0} and \eqref{closenormals} we obtain that $u$ is small in the $C^1$-norm. 

Now we prove the claim. Let $q \in \Sigma$ be fixed, and let $z \in \W$ be given so that $q=\eta(z)$. From \eqref{pC0} we know that 
\[
\abs{q - p(q)} \le \epsilon,
\]
and from \eqref{etaC1estimate} we know that 
\begin{equation}\label{lastcloseness}
\abs{q - z}\leq \epsilon,\qquad \abs{\nu_\Sigma(q) - \nu_\W(z)} \le \epsilon.
\end{equation}
Patching the inequalities together, we get 
\[
\abs{p(q) - z} \le 2\epsilon.
\]
Since the Wulff shape is convex, necessarily $z$ must belong to a graph parametrisation $\phi_{p(q)} \fromto{\B^n_R}{\W}$ centered in $p(q)$, provided that $\epsilon>0$ is sufficiently small. By convexity, we easily notice that 
\[
\abs{\nu_\W(p(q)) - \nu_\W(z)} \le c(n, \, \W) \epsilon.
\]
Patching this inequality with \eqref{lastcloseness} we obtain the claim, and we conclude the proof of Proposition \ref{QualProp4Ch}.

\subsection*{Conclusion}
Recalling Definition \ref{deff} and combining Proposition \ref{QualProp4Ch}, Theorem \ref{MainEstimateb} and \cite[Proposition 7.1]{GiofDeR}, we obtain the following result.
\begin{prop}\label{AlmostMainThm}
Under the hypothesis of Proposition \ref{QualProp4Ch}, we have the additional estimate:
 \begin{equation}\label{MainEstimateEq4Ch}
 \norm{u - \phi_{v_u}}_{W^{2, \, p}(\W)} \le C \coup*{ \norm{\Sdot_F(\Sigma)}_{L^p(\Sigma)} + \epsilon \norm{u}_{W^{2, \, p}(\W)}}.
 \end{equation}
 where $C=C(n, \, p, \, \W)$.
\end{prop}

We end the section by getting rid of the function $\phi_{v_u}$ in estimate \eqref{MainEstimateEq4Ch}, that is, proving the following:

\begin{prop}\label{CenterProp4Ch}
Let $\Sigma$ be a closed hypersurface in $\R^{n+1}$ satisfying \eqref{assumption} and \eqref{DeltaAnisPinching4Ch}, so that the estimates of Propositions \ref{QualProp4Ch} and \ref{AlmostMainThm} hold for a radial anisotropic parametrization $\psi$.
There exist $\epsilon_0>0$, $C_0>0$ depending only on $\W$ with the following property. If \eqref{QualThmEst} holds with $\epsilon \le \epsilon_0$, then there exists $c=c(\Sigma) \in \R^{n+1}$ such that  $\Sigma - c$ still admits a radial parametrization 
\[
\psi_c \fromto{\W}{\Sigma-c},\quad \psi_c(x):=x+u_c(x)\nu_{\W}(x),
\]
and $u_c$ satisfies:
\[
\begin{cases}
\|u_c\|_{C^1} \le C_0\epsilon, \\
\langle u_c, \, \phi_w\rangle_{L^2} = 0 \quad \mbox{ for every } \phi_w \text{ defined as in \eqref{PhiDef}}.
\end{cases}
\]
\end{prop}

\begin{proof}
We divide the proof into three main steps. 
\paragraph*{Step 1} 
\textit{For any positive constant $C_1$ there exist positive numbers $\epsilon$, $C_2$ depending only on $\W$, $C_1$ with the following property. 
For every $c \in \B_{C_1 \epsilon}^{n+1}$, the hypersurface $\Sigma_c:= \Sigma - c$ is still a graph over $\W$, and its radius $u_c$ satisfies}
\[
\|u_c\|_{C^1(\W)} \le C_2 \epsilon.
\]
We consider $\epsilon$ so small that $\Sigma_c$ is still in the $2\epsilon$-tubular neighborhood of $\W$.  Again, we argue by proving that the projection map 
\[
p_c \fromto{\Sigma_c}{\W},\quad p_c \colon q=x+u_c\nu_{\W}(x) \longmapsto x
\]
is a diffeomorphism. 
Following the same strategy of the proof of Proposition \ref{QualProp4Ch}, we just need to show that $\nu_{\W}(p_c(q)) \notin T_q \Sigma_c$ for every $q \in \Sigma_c$. Let then $q \in \Sigma_c$ be given. By the very definition of $\Sigma_c$, we have that $\tilde{q}:= q -c  \in \Sigma$. Moreover, since $\Sigma$ is a graph over $\W$ with radius $u$, there exists $x \in \W$ such that $\tilde{q} =x+u(x)\nu_{\W}(x)$. By the computation made in \cite[App. B]{GiofDeR}, we deduce
\begin{equation}\label{stima1}
\left| \nu_\Sigma(\tilde{q}) - \nu_{\W}(x)\right| \le C(\W)\epsilon.
\end{equation}
Since $\Sigma_c=\Sigma+c$, we know that $\nu_\Sigma(\tilde{q} ) = \nu_{\Sigma_c}(\tilde{q} + c) = \nu_{\Sigma_c}(q)$. On the other hand,
\begin{equation}\label{stima2}
\left|  \nu_{\W}(p_c(q)) - \nu_{\W}(x) \right| \le \epsilon.
\end{equation}
Combining \eqref{stima1} with \eqref{stima2}, we deduce that 
\begin{align*}
\left| \nu_{\Sigma_c}(q) - \nu_{\W}(p_c(q)) \right| = \left| \nu_\Sigma(\tilde{q}) - \nu_{\W}(p_c(q)) \right| \le  \left| \nu_\Sigma(\tilde{q}) - \nu_{\W}(x)\right|+ \left| \nu_{\W}(x) - \nu_{\W}(p_c(q)) \right| \overset{\eqref{stima1},\eqref{stima2}}{\leq} C\epsilon.
\end{align*}

This shows that for $\epsilon$ sufficiently small, $\nu_{\W}(p_c(q)) \notin T_q\Sigma_c$, and thus we can conclude as in the proof of Theorem \ref{QualProp4Ch}. 

\paragraph*{Step 2} 
\textit{We consider the map
\begin{equation*}\label{CenterPhiDef}
\Phi \fromto{\B_{C_1 \epsilon}^{n+1}}{\R^{n+1}},\quad \Phi(c):= \sum_{i=1}^n \langle u_c, \, \phi_{w_i}\rangle_{L^2} w_i
\end{equation*}
where $\phi_{w_i}$, $w_i$ are defined in Definition \eqref{deff}.
Then there exists a constant $C_3$ depending on $C_1$ such that the following estimate holds:}
\begin{equation}\label{ClosePhiEst}
\left| \Phi(c) - \Phi(0) - c\right| \le C_3 \epsilon^2.
\end{equation}
Indeed, for every $c$ such that $|c| < C_1 \, \epsilon$ we find 
\[
d_{\hd}(\Sigma -c, \, \W) \le d_{\hd}(\Sigma -c, \, \Sigma) + d_{\hd}(\Sigma, \, \W) \le (C_1+1) \epsilon.
\]
Arguing as in the Step 1, it is easy to see that also the function $u_c$ satisfies the estimates 
\begin{equation}
\|u_c\|_{C^1} \le C(n, \, \W) \epsilon, \label{14Ch}
\end{equation}
 We start the linearisation with the following simple consideration: for every $z \in \W$ there exists $x_c=x_c(z)\in \W$ so that
\[
\psi_c(z) = \psi(x_c(z)) - c.
\]
We expand this equality and find
\begin{equation}\label{Equalityzxc4Ch}
 z + u_c (z) \, \nu_{\W}(z) = x_c(z) + u(x_c(z)) \nu_{\W}(x_c(z)) - c .
\end{equation}
Using the $C^0$-smallness of $u$ and $u_c$, we can easily see that $x_c=x_c(z)$ satisfies the relation
\begin{equation*}\label{xcApp4Ch}
\abs*{x_c(z) - z} \le C\coup*{n, \, \W}\epsilon .
\end{equation*}
This approximation, combined with \eqref{14Ch}, gives an estimate of $u$ close to $z$: 
\begin{equation}\label{Estimate324Ch}
\abs*{u(x_c(z)) -  u(z)} \le C(n, \, \W)\epsilon^{2}.
\end{equation}
We evaluate $F^*$ in the point in \eqref{Equalityzxc4Ch}:
\begin{align*}
\underbrace{F^*( z + u_c (z) \, \nu_{\W}(z))}_{= 1 +  u_c(z) \restr{ dF^*}{z}[\nu_{\W}(z)] + \mR } = 
\underbrace{F^*( x_c(z) + u (x_c(z)) \, \nu_{\W}(x_c(z)) -c)}_{= 1 +  u(x_c(z)) \restr{ dF^*}{x_c(z)}[\nu_{\W}(x_c(z))] - \restr{ dF^*}{x_c(z)}[c]  + \mR },
\end{align*}
where 
\[
\abs*{\mR} \le C(n, \, \W) \epsilon^2.
\]
Plugging in the previous equality the gauge property \eqref{Robin}, we obtain
\[
 \abs*{u_c(z)\langle \nu_{\W}(z), \, \nu_{\W}(z)\rangle - u(x_c(z))\langle \nu_{\W}(x_c), \, \nu_{\W}(x_c)\rangle + \langle c, \, \nu_{\W}(x_c)\rangle } \le C(n, \, \W) \epsilon^{2},
 \]
which by \eqref{Estimate324Ch} reads
\begin{equation}\label{GoodEstimate4Ch}
\abs{u_c(z) - u(z) + \underbrace{\langle c, \, \nu_{\W}(z)\rangle}_{=\phi_c(z)}}  
\le C(n, \, \W) \epsilon^{2}.
\end{equation} 
Integrating over $\W$ and using \eqref{GoodEstimate4Ch}, we conclude the proof of Step 2.

\paragraph*{Step 3} \textit{Conclusion}. 
In order to obtain the thesis, we will prove the following claim (the proof of Claim is postponed right after this proof):
{\em
\begin{itemize}
\item[\textit{\bf Claim}]Let $G$ be a continuous map $ G \daA{\B^{n+1}_{1}}{\erre^{n+1}}$ which satisfies the estimate 
\begin{equation}\label{PhiLinEst3Ch}
\abs{G(x) - a - x} \le \epsilon \quad \forall x \in \B^{n+1}_{1} \qquad  \mbox{ with } \abs{a} < \frac{1}{10}.
\end{equation}
Then $G$ must have $0$ in its image if $\epsilon$ is sufficiently small.
\end{itemize}}
%
%

 This claim gives us the thesis since we can always reduce to this case by choosing a $C_1$ big enough (depending only on $n$ and $\W$) and via a proper rescaling. Indeed, we can define 
 \[
\tilde{\Phi} \fromto{\B_1^{n+1}}{\R^{n+1}}, \tilde\Phi(c):= \frac{\Phi(  C_1 \epsilon c)}{ C_1 \epsilon}
\]
and observe that
\[
\abs*{\tilde{\Phi}(c) - \frac{\Phi(0)}{C_1 \epsilon} - c} = \frac{1}{C_1 \epsilon} \abs*{\Phi(C_1\epsilon c) - \Phi(0) - C_1\epsilon c} \overset{\eqref{ClosePhiEst}}{\le} \frac{C(n, \, \W )\epsilon }{C_1} .
\]
Moreover, 
\[
\frac{\abs{\Phi(0)}}{C_1 \epsilon} \epsilon \le \frac{C(n, \, \W)}{C_1} \le \frac{1}{10}
\]
if we choose the proper $C_1(n, \, \W)$. Therefore, by the claim, we can find $\tilde{c} \in \B^{n+1}_1$ such that $\tilde{\Phi}(\tilde{c})=0$, i.e. $\Phi(C_1 \epsilon \tilde{c})=0$, and we have finished.
\end{proof}

\begin{proof}[Proof of Claim]
We argue by contradiction, and assume that $0$ is not in the image of $G$. Therefore, the rescaled map 
\[
\xi := \frac{G}{\abs{G}} \daA{\B^{n+1}_1}{\S^n} 
\]
is well defined. Now, we know that $G$ satisfies \eqref{PhiLinEst3Ch}. Thus, we obtain:
\begin{align}
\abs{G(x)}^2 = \abs{a + x}^2 + \abs{G(x) - a - x}^2 + 2\scal{a+x, \, G(x) - a - x}= 1 + \abs{a}^2 + 2\scal{a, \, x} + \mR, \label{PreliminaryRescal}
\end{align}
where $\abs{\mR} \le C(n, \, \W)\epsilon$. From \eqref{PreliminaryRescal} we have:
\begin{equation}\label{PreliminaryRescal2}
\frac{79}{100} - C(n, \, \W)\epsilon \le \abs{G(x)}^2 \le \frac{121}{100} + C(n, \, \W)\epsilon.
\end{equation}
We use inequalities \eqref{PhiLinEst3Ch} \eqref{PreliminaryRescal} and \eqref{PreliminaryRescal2} to infer the following estimate:
\begin{align*}
\abs{\xi(x) - x}
&= \abs*{\frac{G(x)}{\abs{G(x)}} - x} = \frac{1}{\abs{G(x)}}\abs{G(x) - \abs{G(x)} x } \\ 
&= \frac{1}{\abs{G(x)}}\abs{G(x) - a - x + a + x(1- \abs{G(x)}) } \le \frac{1}{\abs{G(x)}} \coup*{ \abs{a} + C\epsilon + \abs{1 - \abs{G(x)}} } \\
&\le \frac{10}{\sqrt{79 - C\epsilon}} \coup*{ \frac{1}{10}+ \frac{\sqrt{21}}{10} + C\sqrt{\epsilon} } \le \frac{1 + \sqrt{21}}{\sqrt{79 - C\epsilon}}\coup*{1 + C \sqrt{\epsilon}} \le \frac{\sqrt{2}}{2} + C\sqrt{\epsilon}
\end{align*}
where the constant $C$ depends only on $n$ and $\W$. Therefore, for every $0<\epsilon<1$ sufficiently small, we obtain 
\begin{equation}\label{HomotopyIdentity3Ch}
\abs{\xi(x) - x } < 2 \mbox{ for every } x \in \S^n.
\end{equation}
Therefore the map $ \overline{\xi}:= \restr{\xi}{\esse^n}$ defined as the restriction of $\xi$ to the sphere is well defined.
The thesis follows by a simple application of topological degree theory, which can be found in \cite[Ch.5]{Hirsch}:  since $\overline{\xi}$ is the restriction of a map on the sphere, it must have degree equal to $0$, but \eqref{HomotopyIdentity3Ch} easily implies that $\overline{\xi}$ is homotopic to the identity, and therefore it must have degree equal to $1$, giving the desired contradiction.
\end{proof}
 
\begin{proof}[Proof of Theorem \ref{MainThm}]
We observe that Proposition \ref{CenterProp4Ch} implies that $u_c$ satisfies an improved version of \eqref{MainEstimateEq4Ch}, that is:
$$
 \norm{u_c}_{W^{2, \, p}(\W)} \le C \coup*{ \norm{\Sdot_F(\Sigma)}_{L^p(\Sigma)} + \epsilon \norm{u_c}_{W^{2, \, p}(\W)}}.
$$
This in particular concludes the proof of Theorem \ref{MainThm}.
\end{proof}

\section{Quasi Einstein Hypersurfaces}\label{Ricci}
In this section we focus on the proof of Theorem \ref{Thm1} and Theorem \ref{Thm2}. We first recall the geometric quantities involved:
\subsection*{Geometric quantities}
We fix the sign convention for the main geometric quantities we are going to study in this section. 
We define 
\[
\text{R}(X, \, Y)Z := \nabla^2_{Y, \, X} Z - \nabla^2_{X, \, Y} Z
\]
The Riemann curvature is the $4$-covariant tensor given by lowering one index in the previous expression.
\begin{equation*}\label{RiemannSign}
\Riem(X, \, Y, \, Z, \, W) = \left \langle\text{R}(X, \, Y)Z, \, W\right \rangle
\end{equation*}
The Ricci curvature is the $2$-covariant tensor given by taking the $(1, \, 3)$-trace of the Riemann curvature:
\begin{equation*}\label{RicciSign}
\Ric_{ij} := g^{pq} \Riem_{ipjq}.
\end{equation*}
Finally, the scalar curvature is given by taking the trace of the Ricci curvature:
\begin{equation*}\label{CurvatureSign}
R= g^{ij} \Ric_{ij}.
\end{equation*}
We recall the following well known corollary of the differential Bianchi identity (see \cite[p. $184$]{GHL}), which relates the derivatives of the Ricci curvature with the derivatives of the scalar curvature.
\begin{lemma}\label{DiffBian}
Let $(M, \, g)$ be an $n$-dimensional Riemannian manifold, with $n\ge 3$. The following equation holds:
\begin{equation*}\label{DiffBianEq}
\divv \Ric = \frac{1}{2} \nabla R
\end{equation*}
\end{lemma}
 Moreover, we recall the Gauss equation for hypersurfaces in a Euclidean space (see \cite[Thm 5.5]{GHL}): 
 \begin{teo}\label{GaussThm}
 Let $\Sigma$ be a hypersurface in $\R^{n+1}$. Then the following equation holds:
\begin{equation}\label{GaussEq}
\Riem_{ijkl} = \frac{1}{2} \coup*{h \nomizu h}_{ijkl} = h_{ik} \, h_{jl} - h_{il} \, h_{jk} .
\end{equation}
\end{teo} 
Contracting the indices in \eqref{GaussEq} we obtain 
\begin{equation*}\label{GaussRicci}
\Ric_{ij} = H h_{ij} - h_i^k h_{kj} .
\end{equation*}

%

We now proceed to give the idea of the proof of Theorem \ref{Thm1} and Theorem \ref{Thm2}.
The strategy we would like to use is basically the same as the one used for Theorem \ref{MainCor}, that is:

Let us consider a sequence of hypersurfaces $\coup*{\Sigma_k}_{k \in \N}$ satisfying either \eqref{PinchingStrict} or \eqref{Pinchingc0}, and 
\[
\lim_k \norm{\tRic^k}_{L^p(\Sigma_k)}=0.
\] Firstly, we estimate the diameter of $\Sigma_k$ and consider a (not relabeled) subsequence $\Sigma_k$ that converges in the Hausdorff distance to a subset $\Sigma_0 \subset \R^{n+1}$. \textit{If} $\Sigma_0$ were a smooth manifold, and \textit{if} the decay of the traceless Ricci tensor passed to the limit, than $\Sigma$ would be a smooth, closed Einstein manifold in $\R^{n+1}$, which is necessarily the round sphere. Then, performing a fine analysis of the $\phi_q$, we would obtain than every graph parametrisation of $\Sigma_k$ must converge to the graph parametrisation of the sphere, and thus we could build the same proof made for Theorem \ref{MainCor}. 

The problem here are the two ifs above, which need to be motivated. First of all, the set $\Sigma_0$ we will find is a priori only a compact subset in $\R^{n+1}$; moreover, the Ricci operator is not elliptic when viewed as a differential operator acting on the function which describes $\Sigma$ as a graph parametrisation. Also if we consider the Gauss equation \eqref{GaussEq} and consider the associated polynomial equation for the eigenvalues of $h$, then the equality
\[
\Ric = (n-1)\lambda 
\]
implies $h = \lambda \Id$ only when $\lambda>0$. Thus, we also need to prove the positivity of $\lambda$ in order to achieve our result. 

We split the proof of the qualitative closeness into two main propositions.

\begin{prop}\label{Harmonicweak4Ch}
Let $\phi_k \fromto{\B^n_R}{\R^{n+1}}$ be a sequence of graph parametrizations, and let then $\Graph(u_k, \, \B^n_R)$ be their image. Assume that every $u_k$ satisfies the following properties:
\begin{itemize}
\item $u_k(0)=0$, $\partial u_k(0)=0$;
\item $\norm{u_k}_{W^{2, \, p}} \le c_0$;
\item $u_k \debto u_0$ weakly in $W^{2, \, p}$;
\item the sequence of hypersurfaces $\coup*{\Graph(u_k, \, B^n_R)}_{k \in \N}$ satisfies
\[
\lim_k \norm{\Ric^k - (n-1)\lambda_0 g^k}_{L^p} = 0.
\] 
\end{itemize}   
Then there exists a radius $0<\rho_0=\rho_0(n, \, p, \, c_0)$ such that the function $u_0$ is smooth (actually analytic) in $\B^n_{\rho_0}$, and the hypersurface $\Graph(u_0, \, \B^n_{\rho_0})$ is Einstein.
\end{prop}

\begin{prop}\label{MeanRicciPositive4Ch}
For every $\epsilon>0$ there exists $0<\delta=\delta(n, \, p, \, c_0, \, \epsilon)$ with the following property. 

Let $\Sigma$ be a closed hypersurface in $\R^{n+1}$ satisfying either \eqref{PinchingStrict} or \eqref{Pinchingc0}. If $\norm{\Ric - (n-1) \lambda_0 g}_{L^p} \le \delta$, then $\lambda_0 > 0$, and for every $q \in \Sigma$, the graph parametrisation $\phi_q$ satisfies:
\begin{equation*}
\norm*{u_q(\cdot) - \mu_0^{-1} \coup*{ \sqrt{1 - \mu_0^2 \abs{ \, \cdot \, }^2} - 1 }}_{C^1} \le \epsilon,
\end{equation*}
where $\mu_0 = \sqrt{\lambda_0}$.
\end{prop}

Combining these two propositions, we obtain the $C^1$-closeness, and then we show how to conclude, proving an improved oscillation result:
\begin{prop}\label{RicciOscillationLp2Ch}
Let $n\ge 2$ be given, ad let $\Sigma$ be a closed hypersurface in $\R^{n+1}$  such that $\Vol_n(\Sigma)=\Vol_n(\S^n)$. Assume $\Sigma$ satisfies one of two following hypothesis:
\begin{itemize}
\item[$a)$] $\Sigma$ is convex, and $ \norm{h}_{L^p(\Sigma)} \le c_0$ for some $1<p<\infty$.
\item[$b)$] $\fint_{\Sigma }\Scal=:\overline{\Scal}>0$ and $\norm{h}_{L^p(\Sigma)} \le c_0$ for some $n<p<\infty$. 
\end{itemize}
Then the following inequality holds.
\begin{equation}\label{RicciOscillationLpEq2Ch}
\norm*{\Riem - \frac{\overline{\Scal}}{2n(n-1)} g \nomizu g }_{L^p(\Sigma)} \le C(n, \, p, \, c_0) \norm{\tRic}_{L^p(\Sigma)}.
\end{equation}
\end{prop}

\subsection{Proof of the $C^1$-closeness}
We start by proving the first proposition.

\begin{proof}[Proof of Proposition \ref{Harmonicweak4Ch}]
The proof uses the concept of \textit{harmonic coordinates}. We recall the definition: given a manifold $(M, \, g)$ and an open set $U \subset M $ a mapping $y \daA{U}{\erre^{n+1}}$ is said to be a \textit{harmonic chart} if it is a diffeomorphism and if it satisfies the equation 
\[
\Delta_g y = 0.
\]
The functions $y^1, \, \dots \, y^n$ are called \textit{harmonic coordinates}. A detailed study on the topic can be found in \cite[Sec. 8.10, p.523]{Jost} or \cite[Ch. 10, Sec. 2.3]{PetersenBook}.
Harmonic coordinates have several properties which make them very suitable for our problem. Indeed, the following expression holds:
\begin{equation}\label{RicciHarm}
-\frac{1}{2} \Delta_{g} g_{ij} + Q_{ij}(g, \, \partial g) = \Ric^{g}_{ij} \mbox{ for every indices } i, \, j,
\end{equation}
where $g_{ij}:= g\coup*{\frac{\partial}{\partial y^i}, \, \frac{\partial}{\partial y^j}}$, $Q_{ij}$ is a universal polynomial depending on $g$ and its first derivatives $ \partial g$. The computations can be found in \cite[Ch. 10, Sec. 2.3]{PetersenBook}.

In the aforementioned references however, the authors work under stronger regularity assumptions on the metric. In our case we ought to perform a finer study. We prove the following result.

\begin{lemma}\label{LemmaHarm4Ch}
Let $u \fromto{\B^n_R}{\R}$ be given so that $u(0)=0$, $\partial u(0)=0$, $\norm{u}_{W^{2, \, p}(\B^n_R)} \le c_0$.  Set $G_\rho:=\Graph(u, \, \B^n_\rho)$ for $0<\rho \le R$. Then there exist $0<\rho_0=\rho_0(n, \, p, \, c_0)$ and a diffeomorphism $\eta \fromto{G_{\rho_0}}{\R^n}$ such that 
\[
\Delta_g \eta  = 0,\ \norm{\eta}_{W^{2, \, p}(G_{\rho_0})} \le c_0,
\]
with $\Delta_g$ being the Laplace-Beltrami operator associate to the manifold $G_{\rho_0}$.
\end{lemma}
\begin{proof}
By pull-back we work on the sequence $\coup*{ \B^n_R, \, g^k }_{k \in \N}$, with $g^k = \delta + \partial u_k \otimes \partial u_k$.
We are going to show the existence of a $0<\rho_0=(n, \, p, \, c_0)<R$ such that the map $\eta \fromto{\B^n_{\rho_0}}{\R^n}$ defined by
\[
\begin{cases}
\Delta_g \eta = 0 \mbox{ in } \B_{\rho_0}^n, \\
\restr{\eta}{\partial \B_{\rho_0}^n} = x
\end{cases}
\]
is a diffeomorphism and satisfies $\norm{\eta}_{W^{2, \, p}} \le c_0$. In order to simplify the proof, we will consider a rescaled version the problem.
Firstly, let us recall the expression in chart of the Laplace-Beltrami operator:
\begin{equation}\label{LaplaceBeltrami}
\Delta_g = \frac{1}{\sqrt{\det g}} \partial_i \coup*{ \sqrt{\det g} g^{ij} \partial_j }.
\end{equation} 
Let $\eta \fromto{\B^n_\rho}{\R^n}$ be a map satisfying $\Delta_g \eta = 0$. We say that the map 
\[
\eta_\rho \fromto{\B^n_1}{\R^n},\ \eta_\rho(z):= \frac{\eta(\rho z)}{\rho}
\]
satisfies $\Delta_{g_p} \eta_p = 0$, where $\Delta_{g_\rho}$ is the Laplace-Beltrami operator associated to the metric $g_\rho(z):=g(\rho z)$, defined on the ball $\B^n_1$. Indeed, if we set $a^{ij}:= \sqrt{\det g}g^{ij}$ and $a_\rho^{ij}:= \sqrt{\det g_\rho}g_\rho^{ij}$ , then
\begin{align*}
 \partial_i \coup*{ a^{ij}_\rho \partial_j \eta_\rho} (z) 
&=\partial_i a^{ij}_\rho(z) \partial_j \eta_\rho(z) + a^{ij}_\rho(z) \partial^2_{ij} \eta_\rho (z) \\ 
&=\rho \coup*{ \partial_i a^{ij}(\rho z) \partial_j \eta (\rho z) + a^{ij}(\rho z) \partial^2_{ij} \eta (\rho z)  } = \rho \partial_i \coup*{ a^{ij} \partial_j \eta} (\rho z)=0.
\end{align*}
Moreover, since $g=\delta + \partial u \otimes \partial u$ and $u$ satisfies $u(0)=\abs{\partial u(0)}=0$ and $\norm{u}_{2, \, p} \le c_0$, then we also have
\[
\lim_{\rho \to 0} \norm{g_\rho - \delta}_{W^{1, \, p}(\B^n_1)} = 0.
\]
We have reduced the problem to the following formulation:
\begin{itemize}
\item[{\bf Claim}] There exists $0<\epsilon_0=\epsilon_0(n, \, p)$ with the following property. If $g$ is a metric on $\B^n_1$ such that $\norm{g-\delta}_{W^{1, \, p}} \le \epsilon_0$, then there exists a diffeomorphism $\eta \fromto{\B^n_1}{\R^n}$ such that 
\[
\Delta_g \eta = 0,\ \norm{\eta - \Id}_{W^{2, \, p}} \le \epsilon_0.
\]
\end{itemize}
We now prove the Claim showing that the only solution $\eta$ of the problem
\[
\begin{cases}
\Delta_g \eta = 0, \\
\restr{\eta}{\partial \B_1} = x
\end{cases}
\] 
is a diffeomorphism, provided that $\epsilon_0$ is sufficiently small.
The solution $\eta$ exists and it is smooth, since the coefficients are smooth. We prove that $\eta$ satisfies the aforementioned a priori $W^{2, \, p}$-estimate and is a diffeomorphism in $\B_1$. 
From \eqref{LaplaceBeltrami} we get that our equation is of the divergence form:
\[
\partial_i \coup*{a^{ij} \partial_j \eta} = 0, \mbox{ where } \norm{a^{ij} - \delta^{ij}}_{W^{1, \, p}} \le \epsilon_0.
\]
Since $n<p$, we have that the Sobolev closeness is also a $C^{0, \, \alpha}$-closeness, thus we obtain that for $\epsilon_0$ sufficiently small the matrix $a=a^{ij}$ satisfies the bound
\[
\frac{1}{2} \delta \le a \le 2 \delta
\]
in the sense of quadratic forms. This bound will be useful when we will deal with sequences of metrics converging weakly, because it passes to the limit and triggers the classical elliptic theory for weak solutions (see \cite{GilTrud}).
Now we conclude:
\[
\norm{\eta - x}_{W^{2, \, p}(\B^n_1)} \le C(n, \, p) \norm{\Delta_g \coup*{\eta - x}}_{L^p(\B^n_1)} = C  \norm{\Delta_g x}_{L^p(\B^n_1)},
\]
and therefore
\begin{align*}
\Delta_g x^k  
&= \frac{1}{\sqrt{\det g}} \partial_i \coup*{ \sqrt{\det g} g^{ij} \partial_j x^k } = \frac{1}{\sqrt{\det g}} \partial_i \coup*{ \sqrt{\det g} g^{ik}} \\
&= \partial_i g^{ik} + \tr \coup*{g^{pq} \partial_i g_{pq}} g^{ik}.
\end{align*}
From this computation we obtain
\[
\norm{\eta - x}_{W^{2, \, p}(\B^n_1)} \le C(n, \, p) \norm{g - \delta}_{W^{1, \, p}(\B^n_1)} \le C \epsilon_0,
\]
and for $\epsilon_0$ sufficiently small we obtain the desired Claim.
\end{proof}
We use Lemma \ref{LemmaHarm4Ch} to prove Proposition \ref{Harmonicweak4Ch}. Indeed, let $\phi_k$, $u_k$ and $\Graph(u_k, \, \B^n_R)$ be as in the hypothesis of Proposition \ref{Harmonicweak4Ch}. Using the pull-back, we work in $\coup*{\B^n_R, \, g^k}$. Since $u_k \debto u_0$  weakly in $W^{1, \, p}$ and $n<p$ by hypothesis, we know that $u_k \to u_0$ strongly in $C^{1, \, \alpha}$, then  $\partial u_k \otimes \partial u_k \to \partial u_0 \otimes \partial u_0$ strongly in $C^{1, \, \alpha}$, and 
\begin{equation}\label{weakconv4ch}
 \underbrace{\partial^2 u_k \otimes \partial u_k + \partial u_k \otimes \partial^2 u_k}_{\partial \coup*{\partial u_k \otimes \partial u_k}} \debto  \underbrace{\partial^2 u_0 \otimes \partial u_0 + \partial u_0 \otimes \partial^2 u_0}_{\partial \coup*{\partial u_0 \otimes \partial u_0}} \mbox{ in } L^p.
\end{equation}
From \eqref{weakconv4ch} we deduce  $g^k \debto g^0 = \delta + \partial u_0 \otimes \partial u_0$ in $W^{1, \, p}$. We consider then harmonic coordinates $\eta_k \fromto{\B^n_{\rho_0}}{\R^n}$ built as in Lemma \ref{LemmaHarm4Ch}. From the discussion made above, one can easily infer that $\eta_k$ converges weakly in $W^{2, \, p}$ to the vector valued function $\eta_0$ which is of class $W^{2, \, p}$ and weakly solves the system
\[
\begin{cases}
\Delta_{g^0} \eta_0 = 0, \\
\restr{\eta_0}{\partial \B^n_{\rho_0}} =x,
\end{cases}
\]
Let us now call 
\[
g^k_{ij} := g^k \coup*{ \frac{\partial}{\partial \eta_k^i}, \, \frac{\partial}{\partial \eta_k^j}}, \mbox{ where } \frac{\partial}{\partial \eta^i}:=d \eta \cquad*{ \frac{\partial}{\partial x^i}}.
\] Then every $g^k_{ij}$ solves system \eqref{RicciHarm}, namely
\begin{equation*}
 -\frac{1}{2} \Delta_{g^k} g^k_{ij} + Q_{ij}(g^k, \, \partial g^k) = \Ric^{k}_{ij},
 \end{equation*} 
with $\Ric^k_{ij} := \Ric^k \coup*{ \frac{\partial}{\partial \eta_k^i}, \, \frac{\partial}{\partial \eta_k^j}}$. Then, this equation passes to the limit in  $g^0_{ij}=g^0 \coup*{ \frac{\partial}{\partial \eta_0^i}, \, \frac{\partial}{\partial \eta_0^j}}$, which solves the distributional equation
\begin{equation}\label{EinsteinSystem4Ch}
 -\frac{1}{2} \Delta_{g^0} g^0_{ij} + Q_{ij}(g^0, \, \partial g^0) = \lambda_0 g^0_{ij}.
\end{equation}
Following the computations leading to equation \eqref{RicciHarm} as made in \cite[Sec. 2.3]{PetersenBook}, we can easily notice that the polynomial $Q_{ij}(g^0, \, \partial g^0)$ is of class $L^{p/2}$. Therefore we can apply the bootstrap technique to deduce regularity. 
Indeed, every $g^0_{ij}$ is a $W^{1, \, p}$-weak solution of the equation
\[
L[v]=f,
\] 
where $f$ is a $L^{p/2}$-function. By the Morrey estimates, we know that every $g^0_{ij}$ is actually in $W^{2, \, p/2}$, and in particular 
\[
\partial g^0 \in L^{\coup*{p/2}^*}, \mbox{ where } \coup*{p/2}^* = \frac{n \coup*{p/2} }{n - \coup*{p/2}} = \frac{np}{2n-p}.
\]
A straightforward computation shows
\[
\coup*{p/2}^*>p \Leftrightarrow \frac{np}{2n-p} > p \Leftrightarrow p>n,
\]
and therefore every $Q_{ij}(g^0, \, \partial g^0) \in L^{p_1 / 2}$ for some $p_1=\coup*{p/2}>p$. We proceed inductively until we find $Q_{ij}(g^0, \, \partial g^0) \in L^{p_N/2}$ for some $p_N>2n$. In this case, we obtain that every $g^0_{ij} \in C^{1, \, \alpha}$. At this point, we notice that $Q_{ij}(g^0, \, \partial g^0) \in C^{0, \, \alpha}$. From the Schauder estimates we infer that every $g^0_{ij} \in C^{2, \, \alpha}$, thus rendering $a^{ij}, \, Q_{ij} \in C^{1, \, \alpha}$. Inductively we obtain that $g^0$ is in $C^{k, \, \alpha}$, therefore it is smooth. It can be also proved, that in this context, the metric is actually analytic, and hence we obtain our desired regularity. We refer to \cite{GilTrud} for an overall synthesis on all the aforementioned estimates and elliptic regularity results. Since $g^0$ is regular and satisfies \eqref{EinsteinSystem4Ch}, then the hypersurface $\Graph(u_0, \, \B^n_{\rho_0})$ is Einstein and $\Ric^{0} = \lambda_0 g^0$.
\end{proof}
Now we deal with Proposition \ref{MeanRicciPositive4Ch}

\begin{proof}[Proof of Proposition \ref{MeanRicciPositive4Ch}]
Again we need a useful lemma. 
\begin{lemma}\label{SigmaBound4Ch}
If $\Sigma$ satisfies either \eqref{PinchingStrict} or \eqref{Pinchingc0}, then there exists $0<D_0=D_0(n, \, p, \, c_0)$ such that 
\[
\diam_g \Sigma \le D_0.
\]
\end{lemma}

\begin{proof}
A smooth, closed hypersurface satisfying either \eqref{PinchingStrict} or \eqref{Pinchingc0} allows us to apply Lemma \ref{GraphChartLem} (see Remark \ref{noconv}). Now we consider two points $p_0$, $q \in \Sigma$, such that $d_g(p_0, \, q) = \diam_g(\Sigma)$. Such points clearly exist by compactness. By virtue of Lemma \ref{GraphChartLem} we are able to find $Q$ geodesic balls $B^g_1, \, \dots B^g_Q$, with the following properties: $p_0 \in B^g_1$, $q \in B^g_Q$, $B^g_{i} \cap B^g_{i+1} \neq \varnothing$ and $Q \le N$, where $N=N(n, \, c_0, \, p)$ is the natural number given by Lemma \ref{GraphChartLem}. Then, for every $i=1, \dots, Q-1$ we choose a point $p_i \in B^g_i \cap B^g_{i+1}$, and set $p_Q:=q$. Naturally, since $p_i$, $p_{i+1} \in B^g_i$, the following inequality holds:
\[
 d_g (p_i, q_i) \le 2R.
\]
Then by triangle inequality, we find our desired bound.
\[
\diam_g(\Sigma) = d_g(p_0, \, q) = d_g(p_0, \, p_Q) \le \sum_{i=0}^{Q-1} d_g(p_i, \, p_{i+1}) \le 2QR=D(n, \, p, \, c_0).
\]
\end{proof}
We now come to the proof of Proposition \ref{MeanRicciPositive4Ch}.
Let us argue by compactness, and let $\coup*{\Sigma_k}_{k \in \N}$ be a sequence of closed hypersurfaces satisfying either \eqref{PinchingStrict} or \eqref{Pinchingc0} and satisfying 
\[
\lim_k \norm{\Ric^{g^k} - (n-1)\lambda_0 g^k}_{L^p_k} = 0.
\]
Up to translations, we can assume $b(\Sigma_k)=0$ for every $k \in \N$, where \[
b(\Sigma_k):= \fint_{\Sigma_k} x \, dV_{g^k}(x)
\]
denotes again the barycentre of $\Sigma_k$. Then, by Lemma \ref{SigmaBound4Ch}, the sequence $\coup*{\Sigma_k}_{k \in \N}$ is a sequence of compact sets, all enclosed in a ball. Hence, we can use the classical compactness theorem of Hausdorff to extract a subsequence converging in the Hausdorff distance to a compact set $\Sigma_0 \subset \R^{n+1}$. Let $q_0 \in \Sigma_0$ be a point that attains the maximum distance from $0$, i.e. 
\[
\abs{q_0}^2 = \max_{q \in \Sigma_0} \abs{q}^2.
\]
Let then $\coup*{q_k}_{k \in \N}$ be a sequence of points $q_k \in \Sigma_k$ converging to $q_0$, and $\phi_k$ be the associated graph parametrizations with center $q_k$ and width $R$. Then, up to subsequences, $\phi_k$ converges weakly in $W^{2, \, p}$ to a function $\phi_0 \fromto{\B^n_R}{\R^{n+1}}$.
Since 
\[
\phi_k(z)=q_k + \Phi_k 
\begin{pmatrix}
z \\
u^k(z)
\end{pmatrix},
\] 
it is obvious that $\Phi_k \to \Phi_0$  and $u^k \debto u_0$ weakly in $W^{2, \, p}$. Hence $\phi_0$ is a graph parametrisation, and $\phi_0(0)=q_0$, $\phi_0(\B^n_R) \subset \Sigma_0$. Moreover, since the isometries $\Phi_k$ clearly alter neither the final result nor the proof, we are therefore in the hypothesis of Proposition \ref{Harmonicweak4Ch}, and obtain that $u_0$ is actually smooth and $\phi_0(\B^n_{\rho_0}) = \Graph(u_0, \, \B^n_{\rho}) \subset \Sigma_0$ is a smooth, Einstein manifold. The map  $\phi_0$ has another remarkable property: it satisfies
\[
\abs{\phi(0)}^2 = \abs{q_0}^2 =  \max_{z \in \B^n_{\rho_0}} \abs{\phi(z)}^2.
\]
Deriving twice, we obtain the following equalities holding in $0$:
\[
\underbrace{\scal{\partial_i \phi_0, \, \phi_0(0)}=0}_{\Rightarrow \scal{q_0}^\perp = T_{q_0} \Sigma }, \quad \partial^2 \phi(0) \le 0 \Rightarrow
\underbrace{\scal{ \partial^2_{ij} \phi_0, \, \phi_0(0) }}_{=-\abs{q_0}^{-1} A_{ij}} + 
\underbrace{\scal{ \partial_i \phi_0, \, \partial_j \phi_0 }}_{=g_{ij}} \le 0,
\]
from which we obtain the equality 
\begin{equation}\label{AlmostRicciConvex}
\restr{h}{q_0} \ge \frac{1}{\abs{q_0}} \restr{g}{q_0}.
\end{equation}
Equality \eqref{AlmostRicciConvex} holds just in one point, but it is enough: indeed, $\phi_0(\B^n_{\rho_0})$ is smooth and Einstein, thus at $q_0$ we also have the estimate:
\[
(n-1)\lambda_0 g = \Ric \ge \frac{(n-1)}{\abs{q_0}^2} g \ge \frac{(n-1)}{D_0^2} g \Rightarrow \lambda \ge \frac{1}{D_0^2},
\]
and hence $\lambda_0>0$. Since $\phi_0$ parametrizes an Einstein hypersurface, the equality holds in the whole ball $\B^n_{\rho_0}$. Thanks to Theorem \ref{RigidityRicci2Ch} we obtain that $h=\mu_0 g$, where $\mu_0 = \sqrt{\lambda_0}$. This tells us that $\phi_0$ parametrizes a portion of a round sphere with radius $\mu_0^{-1}$. Since $\phi_k$ converges weakly to $\phi_0$ in $W^{2, \, p}$, we obtain that also the associated function $u^k$ converge to $\mu_0^{-1} \coup*{ \sqrt{1 - \mu_0^2 \abs{x}^2 } - 1 }$ weakly in $W^{2, \, p}$. Since $n<p$ the convergence is also strong in $C^{1, \, \alpha}$. The study we made insofar works not only for $\phi_0$ but for every possible parametrization: let us go back to our sequence  $\coup*{\Sigma_k}_{k \in \N}$ of closed hypersurface. Now we know that $\lambda_0>0$, and thus for every sequence $q_k \in \Sigma_k$, for every $\phi_k$ graph parametrisation  with center $q_k$ and width $\rho_0$, we obtain that every weak limit must parametrize a portion of a sphere with radius $\S^n_{\mu_0^{-1}}$ with $u_0(x)=\mu_0^{-1} \coup*{ \sqrt{1 - \mu_0^2 \abs{x}^2 } - 1 }$ as parametrization, and the convergence is strong in $C^1$. This proves the proposition. 
\end{proof}

Now we repeat the very same passages made in the proof of Theorem \ref{MainCor}, and we easily obtain the corollary:
\begin{cor}\label{RicciCloseCh1}
For every $0<\epsilon$ there exists $0<\delta=\delta(n, \, p, \, c_0, \, \epsilon)$ with the following property. 

Let $\Sigma$ be a closed hypersurface in $\R^{n+1}$ satisfying either \eqref{PinchingStrict} or \eqref{Pinchingc0}. If $\norm{\tRic}_{L^p(\Sigma)}=0$, then there exists a vector $c \in \R^{n+1}$ such that $b(\Sigma-c)=0$, and the radial parametrization
\[
\psi \fromto{\S^n}{\Sigma},\ \psi(x)=e^{f(x)} x 
\]
is well defined. Moreover $\norm{f}_{C^1} \le \epsilon$.
\end{cor}
This concludes the study of the qualitative $C^1$-closeness.

\subsection{The oscillation argument: proof of Proposition \ref{RicciOscillationLp2Ch}}\label{SectionProof252Ch}
The proof of Proposition \ref{RicciOscillationLp2Ch} relies on 
the following oscillation lemma, whose $L^2$-version has been proved under weaker assumptions in \cite{DLT}.
\begin{lemma}\label{ScalOscillationLp2Ch}
Let $\Sigma$ be a closed, convex hypersurface in $\R^{n+1}$. Assume $\Sigma$ satisfies condition $a)$ or condition $b)$ as in Proposition  \eqref{RicciOscillationLp2Ch}. In the latter one, the positivity assumption of $\overline{\Scal}$ is not required. Then the following inequality holds.
\begin{equation*}\label{ScaliOscillationLpEq2Ch}
\norm*{\Scal - \overline{\Scal}}_{L^p(\Sigma)} \le C(n, \, p, \, c_0) \norm{\tRic}_{L^p(\Sigma)}.
\end{equation*}
\end{lemma}
We prove Lemma \ref{ScalOscillationLp2Ch} in Appendix \ref{AppendixOsc}.
From Lemma \ref{ScalOscillationLp2Ch} we wish to derive Proposition \ref{RicciOscillationLp2Ch}. Since the second fundamental form is a symmetric tensor, we know by the spectral theorem that it is diagonalizable. Let $\lambda_1, \, \dots \lambda_n$ be its eigenvalues. Our idea to use equation \eqref{GaussEq} in order to interpret \eqref{RicciOscillationLpEq2Ch} as a polynomial inequality involving the eigenvalues of $h$. Let us define indeed the polynomials
\begin{align}
p(\lambda)
&= \frac{1}{4}\abs*{ h(\lambda) \nomizu h(\lambda) - \kappa \delta \nomizu \delta }^2 = \sum_{i \neq j} \coup*{\lambda_i \lambda_j - \kappa }^2, \label{PolyRiemann} \\
q(\lambda)
&= \abs*{H(\lambda) h(\lambda) - h(\lambda)^2 - (n-1) \kappa \delta}^2 = \sum_{i} \coup*{\lambda_i \coup*{\sum_{i \neq j} \lambda_j} - (n-1) \kappa}^2. \label{PolyRicci}
\end{align}
Here $\kappa \in \R$, and in general we choose it so that $n(n-1)\kappa = \overline{\Scal}$.
Using this notation we can let Proposition \ref{RicciOscillationLp2Ch} easily follow from the following lemma. 
\begin{lemma}\label{PolyLemma2Ch}
Let $0 < \kappa$ be given. Then there exist constants $c_1$, $c_2$, depending on $n$ such that 
\begin{equation*}\label{PolyIneq}
c_1 \le \frac{p(\lambda)}{q(\lambda)} \le c_2, \mbox{ for any } \lambda \in \R.
\end{equation*}
\end{lemma}

From Lemma \ref{PolyLemma2Ch} we easily conclude by integrating the inequality for the eigenvalues of $h$. Indeed, if the mean of the scalar curvature $\overline{\Scal}$ is positive, then from Lemmas \ref{ScalOscillationLp2Ch} and \ref{PolyLemma2Ch} we obtain:
\begin{align*}
\norm*{\Riem - \frac{\overline{\Scal}}{2n(n-1)} g \nomizu g }_p =
\norm*{\Riem - \frac{\kappa}{2} \, g \nomizu g}_p 
\le C(n, \, p) \norm{\Ric - (n-1)\kappa g}_p \le C(n, \, p, \, c_0)\norm{\tRic}_p.
\end{align*}

The positivity of the quantity $\overline{\Scal}$ is easily recovered also in case a): it is straightforward to prove that closed, convex and smooth manifolds have positive mean of the scalar curvature. This quantity is trivially non-negative since we have the formula 
\[
\Scal = \sum_{i \neq j} \lambda_i \lambda_j,
\]
and all the $\lambda_i$ are non-negative by convexity. Let us show that the quantity $\overline{\Scal}$ is actually positive. We consider the function
\[
\xi \colon p \in \Sigma \longmapsto \abs{p}^2.
\]
Let $p_0$ be a maximum for $\xi$, and $\phi_0 \fromto{\B^n_R}{\Sigma}$ be a graph parametrisation around $p_0$, i.e. $\phi_0(0)=p_0$. Since $p_0$ is the maximum of $\xi$, we notice that $\phi_0$ satisfies
\[
\abs{\phi(0)}^2 = \abs{p_0}^2 =  \max_{z \in \B^n_{\rho_0}} \abs{\phi(z)}^2.
\]
Deriving twice, we obtain the following equalities holding in $0$:
\[
\underbrace{\scal{\partial_i \phi_0, \, \phi_0(0)}=0}_{\Rightarrow \scal{p_0}^\perp = T_{p_0} \Sigma }, \quad \partial^2 \phi(0) \le 0 \Rightarrow
\underbrace{\scal{ \partial^2_{ij} \phi_0, \, \phi_0(0) }}_{=-\abs{p_0}^{-1} h_{ij}} + 
\underbrace{\scal{ \partial_i \phi_0, \, \partial_j \phi_0 }}_{=g_{ij}} \le 0,
\]
from which we obtain the equality 
\[
\restr{h}{p_0} \ge \frac{1}{\abs{p_0}} g.
\]
Thus the function $\Scal=\sum_{i \neq j} \lambda_i \lambda_j$ is non-negative and positive in a neighbourhood of $p_0$, hence $\overline{\Scal}>0$.

Let us prove Lemma \ref{PolyLemma2Ch} and conclude.
\begin{proof}[Proof of Lemma \ref{PolyLemma2Ch}]
We first show that the polynomials $p$ and $q$ defined by \eqref{PolyRiemann} and \eqref{PolyRicci}  have the same zeros. Let $Z(p) := \Set{p = 0}$ and $Z(q):=\Set{q = 0}$ be the zero sets  of $p$, $q$, respectively. We claim that:
\begin{equation}\label{ZeroChar2Ch}
Z(p) = Z(q)= \Set{ \sqrt{\kappa} e, \, -\sqrt{\kappa} e}, \mbox{ where } e:=\sum_{i=1}^n e_i.
\end{equation}
We split the proof of Lemma \ref{PolyLemma2Ch} into four main parts. In the first two parts we prove Claim \eqref{ZeroChar2Ch} for $p$ and $q$ respectively. In the third part we study the behaviour of the ratio $p/q$ as $\abs{\lambda}$ approaches $\infty$. In the fourth part we study the behaviour of $p/r$ as $\lambda \to \pm \sqrt{\kappa} e$. From this analysis the lemma will easily follow.

\subparagraph*{Zeros of $p$}
Let $\lambda= \coup*{\lambda_1, \, \dots \lambda_n}$ be given so that $p(\lambda)=0$. Since $p$ is a sum of squares, we get:
\begin{equation}\label{PolyRiemCond}
\lambda_i \lambda_j = \kappa, \mbox{ for every } i \neq j.
\end{equation}
Since $0<\kappa$ we also know that $\lambda_i \neq 0$ for every $i$. Then, for every $i \neq j \neq k$ we immediately find:
\[
\lambda_i \lambda_j = \lambda_j \lambda_k \Rightarrow \lambda_j = \lambda_k=:t,
\]
from which we deduce $\lambda= t e$ for some $t \neq 0$. From \eqref{PolyRiemCond} we immediately deduce $t^2=\kappa$ and the thesis.

\subparagraph*{Zeros of $q$}
Let $\lambda= \coup*{\lambda_1, \, \dots \lambda_n}$ be given so that $q(\lambda)=0$. Since $q$ is a sum of squares, we infer the following system:
\begin{equation}\label{PolyRicci2}
H \lambda_i - \lambda_i^2 = (n-1) \kappa, \mbox{ for every } i,
\end{equation}
where we have set 
\[
H := \sum_{i=1}^n \lambda_i = \scal{\lambda, \, e}.
\]
Notice that from \eqref{PolyRicci2} we have that $\lambda_i \neq 0$ for every $i$. Again, we claim that $\lambda_i=\lambda_j$ for every $i$, $j$. If the claim is true, system \eqref{PolyRicci2} for $\lambda=te$ is reduced to 
\[
(n-1)t^2 = (n-1) \kappa,
\]
and this proves our claim. Let us assume by contradiction that there exist two indices $i$, $j$ such that $\lambda_i \neq \lambda_j$. From \eqref{PolyRicci2} we infer
\begin{equation}\label{PolySum}
H \lambda_i - \lambda_i^2 = H \lambda_j - \lambda_j^2 \Rightarrow H \coup*{\lambda_i - \lambda_j} = \lambda_i^2 - \lambda_j^2 \Rightarrow H = \lambda_i + \lambda_j.
\end{equation}
Substituting \eqref{PolySum} in \eqref{PolyRicci2}, we obtain 
\begin{align}
\lambda_i \lambda_j 
&= (n-1)\kappa, \label{PolyProd} \\
(\lambda_i + \lambda_j) \lambda_h - \lambda_h^2 
&= (n-1)\kappa, \mbox{ for every } h \neq i, \, j. \label{Polykap} 
\end{align}
Assume there exists $\lambda_h \neq \lambda_i$. From equalities \eqref{PolyProd} and \eqref{Polykap} we obtain:
\[
\lambda_i \lambda_j = (\lambda_i + \lambda_j) \lambda_h - \lambda_h^2 \Rightarrow \lambda_j \coup*{\lambda_i - \lambda_h} = \lambda_h \coup*{\lambda_i - \lambda_h},
\]
from which we easily infer $\lambda_h=\lambda_j$. Therefore the coefficients $\lambda_1, \, \dots \,  \lambda_n$ of the point $\lambda$ can take at most two different values. Call them $a$ and $b$, and assume $a$ appears $k$ times and $b$ appears $n-k$ times in the coordinates of $\lambda$. From equality \eqref{PolySum} we have 
\[
(k-1) a + (n-k-1) b = 0.
\]
If both $k-1$ and $n-k-1$ are positive, then $a$ and $b$ must have different sign, and equation \eqref{PolyProd} is violated. If one of them is $0$, say $k-1=0$, then we must have $b=0$, but again equation \eqref{PolyProd} would be violated. Hence all the values are equal, and we easily find the thesis. Notice how the estimate fails when $n=2$. In this case, equality \eqref{PolySum} is not useful, and the polynomials $p$ and $q$ degenerate to
\[
p(\lambda) = q(\lambda) = \coup*{\lambda_1 \lambda_2 - \kappa}^2,
\]
and therefore $Z(p)=Z(q) = \Set{(x, \, y) \in \R^2 | xy = \kappa}$.

\subparagraph*{Boundedness at infinity}
Now we show that the ratio $p(\lambda)/q(\lambda)$ is bounded from above and below when $\abs{\lambda}$ attains large values. A simple computation shows:
\[
\liminf_{|\lambda| \to \infty} \frac{p(\lambda)}{q(\lambda)} = \inf_{\lambda \in \S^n} \frac{\sum_{i \neq j} \lambda_i^2 \lambda_j^2 }{\sum_i \lambda_i^2 \left (\sum_{i \neq j} \lambda_j\right)^2}, \qquad
 \limsup_{|\lambda| \to \infty} \frac{p(\lambda)}{q(\lambda)} = \sup_{\lambda \in \S^n} \frac{\sum_{i \neq j} \lambda_i^2 \lambda_j^2 }{\sum_i \lambda_i^2 \left (\sum_{i \neq j} \lambda_j\right)^2}.
\]
Note that this case represents the study of the ratio $p(\lambda)/q(\lambda)$ in the case $\kappa=0$. Let us do the computation.
Firstly, we claim that in this case the zero sets in the sphere of $p$ and $q$ are finite and satisfy 
\[
Z(p) = Z(q) = \Set{\pm e_1, \, \dots \pm e_n}.
\] 
The claim is straightforward for $p$. For q, let us consider a point $\lambda \in \S^n$ so that $q(\lambda)=0$. Keeping the notation used above, we have the equality:
\begin{equation}\label{PolyAgain}
\lambda_i^2 \coup*{H - \lambda_i}^2 = 0, \mbox{ for every i.}
\end{equation}
Since $\lambda \in \S^n$, then there must exist an index $i$ such that $\lambda_i \neq 0$. Therefore $H=\lambda_i$ must hold. If $\lambda_j = 0$ for all indices $j \neq i$, then necessarily $\lambda=\lambda_i e_i$ and $\lambda_i = \pm 1$, as claimed. If $\lambda_j \neq 0$ for some $j$, then the equality $H=\lambda_j$ must hold and hence $\lambda_j=\lambda_i$. We immediately deduce that the set $\Set{\lambda_1, \, \dots \lambda_n}$ can take only the values $0$ and $t$ for some $t \neq 0$, and not all $\lambda_i$ can be $0$ because $\lambda \in \S^n$. Let us assume w.l.o.g. that $\lambda_1= \, \dots =\lambda_k = t$ and $\lambda_{k+1}= \dots = \lambda_n=0$. From this we can write the equation \eqref{PolyAgain} as
\[
k(k-1)t^4=0,
\]
from which we infer $k=1$, and thus the claim. 

We show now how the ratio $p/q$ is bounded near the zeros of $p$ and $q$.
By symmetry, it is enough to consider the limit for $\lambda \to e_1$. Now we write $\mu := \lambda - e_1$, so that we can study the limit as $\mu \to 0$. Denoting $\tilde{p}(\mu):= p(e_1 + \mu)$, $\tilde{q}(\mu) := q(e_1 + \mu)$, we easily obtain 
\begin{align*}
\tilde{p}(\mu) 
&= 2 \sum_{j=2}^n \mu_j^2 + O(\abs{\mu}^3), \quad \tilde{q}(\mu) 
=  \sum_{j=2}^n \mu_j^2 + \left (\sum_{j=2}^n \mu_j\right)^2 + O(\abs{\mu}^3),
\end{align*}
where $O(\abs{\mu}^k)$ is a quantity which satisfies $\abs{O(\abs{\mu}^k)} \le C(n, \, k)\abs{\mu}^k$.
Therefore we can rewrite the ratio as
\[
\frac{\tilde{p}(\mu)}{\tilde{q}(\mu)} = \frac{2 + O(\abs{\mu})}{1 + \mathcal{R}(\mu) + O(\abs{\mu})  },
\]
where $\mathcal{R}$ satisfies
\[
0 \le \mathcal{R}(\mu) = \frac{\coup*{\sum_{j=2}^n \mu_j}^2}{\sum_{j=2}^n \mu_j^2} \le C(n),
\]
from which we easily deduce the upper and lower bounds.

\subparagraph*{Boundedness near the zeros}
We study now the behaviour of the ratio $p(\lambda)/q(\lambda)$ when $\lambda$ approaches the values $\pm \sqrt{\kappa}e$.  Again, by symmetry it is enough to study the limit at $\sqrt{\kappa} e$. We write $\mu := \lambda - \sqrt{\kappa}e$, and define again $\tilde{p}(\mu):= p(\sqrt{\kappa} e_1 + \mu)$, $\tilde{q}(\mu) := q( \sqrt{\kappa} e_1 + \mu)$. A straight computation for $\tilde{p}$ shows:
\begin{align*}
\tilde{p}(\mu)
&=\sum_{i \neq j} \coup*{\coup*{\mu_i + \sqrt{\kappa}} \coup*{\mu_j + \sqrt{\kappa}} - \kappa }^2 = \sum_{i \neq j} \coup*{ \sqrt{\kappa} (\mu_i + \mu_j) + \mu_i \mu_j }^2 \\
&=\kappa \sum_{i \neq j} (\mu_i + \mu_j)^2 + O(\abs{\mu}^3) = \kappa \sum_{i=1}^n \sum_{\substack{j=1 \\ j \neq i}}^n \mu_i^2 + 2 \mu_i \mu_j + \mu_j^2 + O(\abs{\mu}^3) =2\kappa(n-2) \abs{\mu}^2 + O(\abs{\mu}^3).
\end{align*}
For $\tilde{q}$ we have a similar expression:
\begin{align*}
\tilde{q}(\mu)
&=\sum_{i=1}^n \coup*{ \coup*{\mu_i + \sqrt{\kappa}} \sum_{j \neq i} \coup*{\mu_j + \sqrt{\kappa} } - (n-1)\kappa }^2 = \sum_{i=1}^n  \coup*{ \sqrt{\kappa} \coup*{ (n-1) \mu_i + \sum_{j \neq i} \mu_j} + O(\abs{\mu}^2) }^2 \\
&= \sum_{i=1}^n  \coup*{ \sqrt{\kappa} \coup*{ (n-2) \mu_i +H(\mu)} + O(\abs{\mu}^2) }^2 =\kappa \sum_{i=1}^n \coup*{ (n-2) \mu_i + H}^2 + O(\abs{\mu}^2) \\
&= (n-2)^2 \kappa \abs{\mu}^2 + (3n-4) \kappa H(\mu)^2 + O(\abs{\mu}^3),
\end{align*}
where again $H(\mu)=\sum_i \mu$. From these computations we can easily deduce the lemma. Indeed,
\[
\frac{\tilde{p}(\mu)}{\tilde{q}(\mu)} = \frac{2 + O(\abs{\mu})}{n-2 + \mathcal{R}(\mu) + O(\abs{\mu}) },
\]
where 
\[
0 \le \mathcal{R}(\mu) = \frac{(3n-4) H(\mu)^2}{(n-2) \abs{\mu}^2} \le C(n),
\]
and the lemma is proved.
\end{proof}

\subsection{Conclusion}
Corollary \ref{RicciCloseCh1} is not enough to conclude the estimate, because the Ricci operator seen as differential operator on $f$ is not elliptic. 
We shall conclude the proof of Theorem \ref{Thm1} and Theorem \ref{Thm2} with an idea, that reduces it to an application of Theorem \ref{MainCor}.
First of all, let us show an easy corollary of \ref{RicciCloseCh1}.
\begin{cor}\label{closekappa}
Under the hypothesis of Corollary \ref{RicciCloseCh1}, we have the inequality:
\[
\abs{\overline{\Scal}-n(n-1)} \le C(n, \, p, \, c_0) \epsilon.
\]
\end{cor}
\begin{proof}
The proof has basically already been given in \cite[Lemma 3.7]{GiofEin}. Indeed, although we lack the convexity assumption, from the $C^1$-closeness we are still able to obtain \cite[Equation (3.29)]{GiofEin} 
\begin{equation}\label{curvatureestimateCh4}
\Scal = n(n-1) - 2(n-1) \coup*{\Delta_\sigma f + f} + \coup*{\Delta f}^2 - \abs{\nabla^2 f}^2 + \mathcal{R},
\end{equation}
where $\mathcal{R}$ satisfies 
\[
\abs{\mathcal{R}} \le  C \epsilon \, \coup*{\abs{f} + \abs{\nabla f} + \abs{\nabla^2 f} }.
\]
Since 
\[
\abs*{\int_{\S^n} \mathcal{R}} \le C(n, \, p, \, c_0) \epsilon,
\]
we integrate \eqref{curvatureestimateCh4} and obtain the corollary.
\end{proof}

Now we can complete the proof of Theorem \ref{Thm1} and Theorem \ref{Thm2}.

\begin{proof}[Proof of Theorem \ref{Thm1}]
This will follow directly by the proof of Theorem \ref{MainCor}. Indeed, in this case the strict convexity is translated into an inequality between the traceless Ricci and the traceless second fundamental form, that we would not normally have.
Contracting the indices in \eqref{GaussEq}, we have the equation:
\begin{equation*}
\Ric^i_j = H h^i_j - h^i_k h^k_j.
\end{equation*}
Let then $\lambda_1 \le \dots \le \lambda_n$ be the eigenvalues of $h$. Then the Ricci tensor has eigenvalues $\Lambda_1, \, \dots \Lambda_n$ which satisfy the following equality:
\begin{equation*}\label{RicciEigenv}
\Lambda_j = \lambda_j \sum_{j \neq k} \lambda_k, \quad \forall j=1,\dots,n.
\end{equation*}
By assumption \eqref{PinchingStrict}, we know that $\lambda_j \ge \Lambda$ for every $j=1,\dots,n$, and this allows us to perform the following estimate:
\begin{align*}
|\tRic|^2 
&= \sum_{i \neq j} |\Lambda_i - \Lambda_j|^2 = \sum_{i \neq j} \left (\sum_{k \neq i, \, j} \lambda_k\right)^2 \left |\lambda_i - \lambda_j\right |^2\ge (n-2)^2\Lambda^2  \sum_{i \neq j} \left |\lambda_i - \lambda_j\right |^2 = (n-2)^2 \Lambda^2 |\hdot|^2,
\end{align*}

from which we deduce
\begin{equation*}\label{StrictlyAlmost}
 \|\hdot\|_{L^p(\Sigma)} \le C(n, \, p, \, \Lambda) \|\tRic\|_{L^p(\Sigma)}.
\end{equation*}
This shows how in the strictly convex case, having small $L^p$-norm of the traceless Ricci tensor implies having small $L^p$-norm of the traceless second fundamental form. 

We choose $\delta_0$ sufficiently small so that the hypothesis of \ref{MainCor} holds, and thus we find a vector $c=c(\Sigma)$ such that the associated radial parametrization $ \psi \fromto{\S^n}{\Sigma - c}$  satisfies
\[
\left \|\psi - \Id \right \|_{W^{2, \, p}(\S^n)} \le C \| \hdot \|_{L^p} \le C \| \tRic \|_{L^p(\Sigma)},
\]
as desired.
\end{proof}

\begin{proof}[Proof of Theorem \ref{Thm2}]
 Let us write $\Scal=n(n-1) \kappa$, and assume $\abs*{\kappa - 1} \le \frac{1}{2}$. 

We denote with $\lambda_1 \le \dots \le \lambda_n$ the eigenvalues of $h$ and we consider $\kappa := \frac{1}{n(n-1)} \overline{\Scal}$. As proved in Corollary \ref{closekappa}, we can choose $\delta \le \delta_0$ so that $\kappa$ is between $1/2$ and $2$. Then, given Proposition \ref{RicciOscillationLp2Ch}, we rewrite inequality \eqref{RicciOscillationLpEq2Ch} in terms of the eigenvalues of $h$ and obtain 
\begin{equation}\label{lambdaij}
\left \|\lambda_i \lambda_j - \kappa\right \|_{L^p} \le C \|\tRic\|_{L^p},\ \forall \, i \neq j.
\end{equation} 
From \eqref{lambdaij}, we easily infer for every $k=1,\dots,n$
\begin{equation} \label{lambdaijk}
\left \|\lambda_k(\lambda_i - \lambda_j)\right \|_{L^p} \le C \|\tRic\|_{L^p}.
\end{equation}
 Now, for every $0<\Lambda^2< \kappa$, we define 
\begin{equation*} 
E_\Lambda := \set{q \in \Sigma : |\lambda_{n}(q)| > \Lambda}.
\end{equation*}
We use the set $E_\Lambda$ and its complement in order to perform an estimate on the difference $\left |\lambda_i - \lambda_j\right |$. Indeed, since $\lambda_1 \leq \dots \leq \lambda_n \leq \Lambda$ for every $q \in E^c_\Lambda$, we get the bounds
\[
|\kappa - \Lambda^2| \left |E_\Lambda^c\right |^{\frac{1}{p}} \le \left \|\lambda_i \lambda_j - \kappa \right \|_{L^p(E_\Lambda^c)} \le C \|\tRic\|_{L^p},
\]
which hold for every $ i \neq j$ and $0<\Lambda^2 < \kappa$. Thus we have found 
\begin{equation}\label{ElambdaCEst}
\left |E^c_\Lambda\right |^{\frac{1}{p}} \le \frac{C}{|\kappa - \Lambda^2|} \|\tRic\|_{L^p}.
\end{equation}
On the other hand, for any $i, \, j= 1, \, \dots n-1$, $i \neq j$ we find:
\[
\left \|\lambda_i - \lambda_j\right \|_{L^p(E_\Lambda)} \le \frac{1}{\Lambda} \left \| \lambda_n (\lambda_i - \lambda_j)\right \|_{L^p(E_\Lambda)} \overset{\eqref{lambdaijk}}{\le} \frac{C}{\Lambda} \|\tRic\|_{L^p},
\]
which gives us 
\begin{equation}\label{ElambdaEst}
\left \|\lambda_i - \lambda_j\right \|_{L^p(E_\Lambda)} \le \frac{C}{\Lambda} \|\tRic\|_{L^p}.
\end{equation}
Combining \eqref{ElambdaCEst} and \eqref{ElambdaEst} we obtain 
\begin{equation}\label{almost}
\left \|\lambda_i - \lambda_j\right \|_{L^p} \le C \left ( \frac{1}{\Lambda} + \frac{1}{|\kappa - \Lambda^2|}  \right) \|\tRic\|_{L^p}. 
\end{equation}
This estimate holds for every $i \neq j$, $i, \, j = 1, \, \dots n-1$ and for every $0<\Lambda^2< \kappa$. Equation \eqref{almost} is not sufficient to conclude, because it does not give an estimate on the quantity $\left |\lambda_n - \lambda_j\right |$. This is the only quantity that prevents this proof to give a linear estimate in \eqref{Thm2Est}, forcing us to introduce the exponent $\alpha$.
Indeed, to deal with $\left |\lambda_n - \lambda_j\right |$, we define
\begin{equation*}
\tilde{E}_\Lambda := \set{q \in \Sigma | \left |\lambda_{n-1}(q)\right | > \Lambda}.
\end{equation*} 
With the very same considerations used to deduce \eqref{ElambdaCEst}, we obtain
\begin{equation}\label{Etilambda}
\left |\tilde{E}^c_\Lambda\right |^{\frac{1}{p}} \le \frac{C}{\kappa - \Lambda^2} \|\tRic\|_{L^p}.
\end{equation} 
Now we fix $q \in (n, \, p)$. Then, via H\"older inequality we get
\begin{equation}\label{Etilambda1}
\left \|\lambda_n- \lambda_j\right \|_{L^q(\tilde{E}_\Lambda^c)} \le C(n, \, p, \, c_0)\|\tRic\|_{L^p}^\alpha,
\end{equation} 
where $\alpha$ is defined as in Theorem \ref{Thm2}. Combining \eqref{Etilambda} with \eqref{Etilambda1}, we obtain 
\begin{equation}\label{almost1}
\left \|\lambda_n - \lambda_j\right \|_{L^q} \le C \left ( \frac{1}{| \kappa - \Lambda^2 |} + 1\right) \|\tRic\|_{L^p}^\alpha.
\end{equation} 
Choosing $\Lambda= \sqrt{\frac{\kappa}{2}}$ and plugging together \eqref{almost} and \eqref{almost1}, we deduce 
\[
\|\hdot\|_{L^q} \le \frac{C}{\sqrt{\kappa}}\|\tRic\|_{L^p}^\alpha \le \sqrt{2} C \|\tRic\|_{L^p}^\alpha.
\]
We are thus under the assumptions of Theorem \ref{MainCor}, which provide a radial parametrization  $\psi \fromto{\S^n}{\Sigma}$, $\psi = e^f \Id,$  and a vector $c=c(\Sigma)$ such that \ref{Thm2Est} holds.
\end{proof}

\section{Appendix}\label{AppendixOsc}

\subsection*{Proof of Proposition \ref{AnisOscLp3Ch} and Lemma \ref{ScalOscillationLp2Ch}}
We recall the equations we are going to study.
\begin{align}
&\nabla H_F = \divv S_F. \ \label{EquationC}\\
&\nabla \Scal = \frac{2n}{n-2} \divv \tRic, \label{EquationB} 
\end{align}
These equations present clear similarities, since they are all variations of the equation
\[
\nabla \gu = \divv \gf
\]
in a closed manifold. In both cases, an immediate but naive covering argument may show the existence of a number $\lambda$ such that 
\begin{equation}\label{BaseEstimate}
\norm{\gu - \lambda}_{L^p(M)} \le C(M) \norm{\gf}_{L^p(M)}.
\end{equation}
The problem in such argument is that we do not only have to obtain an estimate, but also to keep an eye on the constant $C$, which in our case has to depend only on general parameters. 
We are going to show an improved estimate which is basically \eqref{BaseEstimate}, but gives a better control on the bounding constant. The technique we are going to use has been used and developed in \cite{Daniel}, where the author deals with the isotropic version of equation \eqref{EquationC}. Considered the massive use we are making of this type of estimates and ideas throughout the paper, we have decided to report the proof.
We split it into the following steps.
\begin{itemize}
\item We show by direct computation in graph parametrisation how the two equations can be written as particular cases of a more general lemma.
\item We obtain a local estimate of our desired inequality, with the bounding constant depending on determined parameters.
\item We show how to make the local estimate global without losing the information on the bounding  constant.
\end{itemize}

\subsubsection*{Step 1: Unifying the equations.}
We recall that from Lemma \ref{Computations4Ch}, if $M=\Graph(u, \, \B^n)$ is a smooth graph, then the following formulas hold:
\begin{align}
g_{ij} &= \delta_{ij} + \partial_i u \, \partial_j u \label{gGraphi} \\ 
g^{ij} &= \delta^{ij} - \frac{\partial^i u \, \partial^j u }{1 + \abs{\partial u}^2} \label{ginvGraphi} \\ 
dV_g &= \sqrt{1 + \abs{\partial u}^2} \, dx \label{VolumeGraph} \\ 
\Gammag_{ik}^k &= v^k h_{ij}, \mbox{ where } v= \frac{\partial u }{\sqrt{1 + \abs{\partial u}^2}}. \label{ChristoffelGraph}
\end{align}
We compute the divergence term of equations \eqref{EquationC}, \eqref{EquationB} in graph parametrisation, and notice how this does not depend on Christoffel symbols.
\begin{itemize}
\item[\eqref{EquationC}]
Firstly, we need to prove that equation \eqref{EquationC} holds. This follows from the computation below. For notation simplicity, we drop the subscript $F$ from  $\coup*{S_F}^i_j=S^i_j$ and $\coup*{A_F}^i_j=A^i_j$ :
\begin{align*}
\divv_g S_{k} 
&= \nabla_i S^i_k = \nabla_i \coup*{\restr{A^i_p}{\nu} h^p_k} =  \nabla_i \coup*{\restr{A^i_p}{\nu} } h^p_k + \restr{A^i_p}{\nu} \nabla_i  h^p_k \\ 
&= \restr{D_q A^i_p}{\nu} h^q_i h^p_k + \restr{A^i_p}{\nu} \nabla_i  h^p_k  
= \restr{D_p A^i_q}{\nu} h^q_i h^p_k + \restr{A^i_p}{\nu} \nabla^p  h_{ik} \\  
&= \nabla_k \coup*{ \restr{A^i_p}{\nu} h^p_i } = \nabla_k H_F.
\end{align*}
Now we notice how also the last divergence term can be written as a flat divergence. We find
\begin{align*}
\divv_g S_{ k} 
&= \nabla_i S^i_k  = \partial_i S^i_k + \Gamma_{ip}^i S^p_k - \Gamma^p_{ik} S^i_p = \partial_i S^i_k + v^i h_{ip} S^p_k - v^p h_{ik} S^i_p \\ 
&=   \partial_i S^i_k + v^i h_{ip} A^p_q h^q_k - v^p h_{ik} A^q_p h^i_q = \partial_i S^i_k + v^i A^{pq} \coup* {h_{ip}  h_{qk} - h_{pk}  h_{qi} } \\
&= \partial_i S^i_k + v^i h_{ip}  h_{qk} \coup*{A^{pq} - A^{qp}} = \partial_i S^i_k. 
\end{align*}
\item[\eqref{EquationB}]
We compute the divergence term in equation \eqref{EquationB}. Firstly we compute the divergence of the Ricci tensor:
\begin{align*}
\nabla_i \Ric^i_k
&= \partial_i \Ric^i_k + \Gamma_{ip}^i \Ric^p_k - \Gamma_{ik}^p \Ric^i_p 
= \partial_i \Ric^i_k +v^i h_{ip} \Ric^p_k - v^p h_{ik} \Ric^i_p \\
&= \partial_i \Ric^i_k + v^i h_{ip} \left (H h^p_k - h^p_q h^q_k \right ) - v^p h_{ik} \left ( H h^i_p - h^i_q h^q_p \right ) \\ 
&= \partial_i \Ric^i_k + H \underbrace{ \left (  v^i h_{ip} h^p_k - v^p h_{ik} h^i_p   \right)}_{= v^i h_{ip} h^p_k -   v^i h_{pk} h^p_i  = 0} +  \underbrace{ \left (  v^p h_{ik} h^i_q h^q_p   - v^i h_{ip} h^p_q h^q_k\right )}_{=v^p ( h_{ik} h^i_q h^q_p - h_{pq} h^q_i h^i_k) = 0} = \partial_i \Ric^i_k. 
\end{align*}
Now we write $\tRic^i_j=\Ric^i_j - \frac{\Scal}{n} \delta^i_j$, and notice that $\delta$ is a symmetric tensor. The computation  of it is identical to the previous one, and we are done.
\end{itemize}
Lastly we write in graph chart $\nabla f = \partial f$, since at the first order the Levi-Civita coincides with the classical derivations. These computations show how we have reduced the two problems to the following: 

\begin{lemma}\label{lemmnew}
 Let $M \subset \R^{n+1}$ be a closed hypersurface. Assume $\Sigma$ has fixed volume $V$ and satisfies the assumptions of Lemma \ref{GraphChartLem}, i.e. admits two numbers $L$ and $R$ such that around every $q \in \Sigma$ we can find a chart defined on the ball $\B^n_R$, which is the graph of a smooth, $L$-Lipschitz function $u_q$.  Assume there are $\gu \fromto{M}{\R}$, $\gf \in \Gamma\coup*{T^* M \otimes T^* M}$ that satisfy a differential relation which in every graph parametrisation at every point admits the following form:
\begin{equation}\label{GraphEquationApp}
\partial_k \gu = \partial_i \gf^i_k \ \mbox{ in } \B_R^n.
\end{equation}
Then there exists a $\lambda \in \R$, such that the following estimate holds.
\[
\norm{\gu - \lambda}_{L^p(M)} \le C(n, \, p, \, V, \, R, \, L) \norm{\gf}_{L^p(M)}.
\] 
\end{lemma}
Notice that in both cases we are studying, the manifold $M$ satisfies the assumptions of Lemma \ref{GraphChartLem}, as explained in Remark \ref{noconv}. These will be crucial in the proof.
In the next step, we prove Lemma \ref{lemmnew}.

\subsubsection*{Step 2: Obtaining local estimates: Proof of Lemma \ref{lemmnew}.}
We begin by working in the graph, and observe that $\gu$ has to satisfy the equation:
\[
\Delta_\delta \gu = \partial^k \partial_i \gf^i_k,
\]
where $\Delta_\delta$ is the flat laplacian. 
The estimate for this equation follows by applying the classic Calderon-Zygmund theorem (See 
 \cite[Prop. 1.11]{Daniel} for a detailed proof in this particular case). We find a constant $C_0:=C(n, \, p)$ and a number $\lambda$ such that
\begin{equation}\label{Local}
\norm{\gu - \lambda}_{L^p (\B_{R/2}^n)}  \le C_0 \norm{\gf}_{L^p (\B_R^n)}
\end{equation}
Estimate \eqref{Local} is almost what we want. It is indeed a local estimate, but it concerns all Euclidean quantities. We show how to swap Euclidean measures with manifold metrics, and how to substitute Euclidean balls with geodesic balls.

The first follows easily from equation \eqref{VolumeGraph} and Remark \eqref{noconv}. Since $\Lip(u) \le L$, we obtain indeed 
\[
dx \le  \sqrt{1 + \abs*{\partial u}^2} \, dx  = dV_g \le \sqrt{1 + L^2} \, dx.
\]
Thus the measures are equivalent, and the control constants depend only on $L$. The same constant $L$ controls the switch from the Euclidean metric $\delta$ to the metric $g$.

Now Lemma \ref{GraphChartLem} allows us to pass from Euclidean to geodesic balls and grants our privileged covering of balls. In particular, we obtain the existence of a radius $R$ such that 
\[
\min_{\lambda \in \R }\norm{\gu - \lambda}_{L^p(B_r^g(q))} \le C(n, \, p, \, V, \, L, \, R) \norm{\gf}_{L^p(M)},\qquad \mbox{for every $0 < r \le R$.}
\]

\subsubsection*{Step 3: Making the estimate global.}
To make the estimate global, we follow the technique used in \cite[p. $6$-$7$]{Daniel} and prove the following lemma.
\begin{lemma}\label{Balls}
Let $M$ be a closed manifold, with fixed volume $\Vol_n(M)=V$.
Suppose $\gu \in C^\infty(M)$ has the following property. There is a radius $\rho$ such that for every $x \in M$ the following local estimate is satisfied:
\begin{equation}\label{LocalDue}
\norm{\gu - \lambda(x) }_{L^p(B_{r}^g(x))} \le  \beta,
\end{equation}
where $\lambda(x)$ is a real number depending on $x$, $r \le 2\rho$ and $\beta$ does not depend on $x$. Then there exists $\lambda \in \R$ such that
\[
\norm{\gu - \lambda }_{L^p(M)} \le C(n, \, \rho, \, V) \beta.
\]
\end{lemma}
\begin{proof}
We choose a finite covering of balls $\set{( B_j^g, \, \lambda_j)}_{j=1}^N$ which satisfies the following properties. Every ball $B_j^g$ has radius $2\rho$, estimate \eqref{LocalDue} holds with $\lambda_j$, and for every $j$, $k$ there exists a ball of radius $\rho$ contained in $B_j^g \cap B_k^g$.  

Therefore, given two balls $B^g_j$ and $B^g_k$ whose intersection is non empty, we have: 
\begin{align*}
\abs{\lambda_j - \lambda_k} 
&= \frac{1}{\Vol_n(B^g_j \cap B^g_k)^{\frac{1}{p}}} \norm{\lambda_j - \lambda_k}_{L^p(B^g_j \cap B^g_k)} = \frac{1}{\Vol_n(B^g_j \cap B^g_k)^{\frac{1}{p}}} \norm{\lambda_j - \gu + \gu - \lambda_k}_{L^p(B^g_j \cap B^g_k)} \\
& \le \frac{1}{\Vol_n(B^g_j \cap B^g_k)^{\frac{1}{p}}} \coup*{ \norm{ \gu - \lambda_k}_{L^p(B^g_j \cap B^g_k)} + \norm{\gu - \lambda_k}_{L^p(B^g_j \cap B^g_k)} } \le \frac{2 \, \beta}{\Vol_n(B^g_j \cap B^g_k)^{\frac{1}{p}}}.
\end{align*}
Using the properties of the covering we obtain
\[
\abs{\lambda_j - \lambda_k} \le 2 \Vol_n(B^g_\rho)^{- \frac{1}{p}} \beta.
\]
 Define $\lambda_{\min} := \min_{1 \le j \le n} \lambda_j$ and  $ \lambda_{\max} := \max_{1 \le j \le n} \lambda_j$. Consider a path joining the ball in the cover with $\lambda_{\min}$ to the one with $\lambda_{\max}$. Since this path can cross at most $N$ different balls, we obtain 
\[
\abs{\lambda_{\max} - \lambda_{\min}} \le 2 N  \Vol_n(B^g_\rho)^{- \frac{1}{p}}  \beta = C(n, \, p, \, \rho) \, \beta.
\]
For every $\lambda_{\min} \le \lambda \le  \lambda_{\max}$ we have
\begin{align*}
\norm{\gu - \lambda}_{L^p_\sigma(\esse^n)}
&\le \sum_{j=1}^N \norm{\gu - \lambda}_{L^p_\sigma(B^g_j)} \le \sum_{j=1}^N \norm{\gu - \lambda_j + \lambda_j - \lambda}_{L^p_\sigma(B^g_j)} \\
&\le \sum_{j=1}^N \norm{\gu - \lambda_j}_{L^p_\sigma(B^g_j)} + \abs{\lambda_j - \lambda} \Vol_n(B^g_j)^{-\frac{1}{p}} \\
&\le \sum_{j=1}^N \norm{\gu - \lambda_j}_{L^p_\sigma(B^g_j)} + \abs{\lambda_{\max} - \lambda_{\min}} \Vol_n(B^g_j)^{-\frac{1}{p}} \le C_2(n, \, p, \, \rho) \, \beta
\end{align*}
and the proof of Lemma \ref{Balls} is completed.
\end{proof}

\subsection*{Acknowledgements}
Antonio De Rosa has been partially supported by the NSF DMS Grant No.~1906451 and the NSF DMS Grant No.~2112311.

\end{document}